\RequirePackage{tikz}
\documentclass[pdflatex,sn-mathphys]{sn-jnl}

\usepackage{subfig}
\usepackage{tikz}
\usetikzlibrary{decorations.markings}
\usetikzlibrary{calc,positioning,decorations.pathmorphing,decorations.pathreplacing,fixedpointarithmetic}
\usepackage{figs/figstyle}
\usepackage{program}
\usepackage{verbatim}
\usepackage{enumerate}
\usepackage{enumitem}

\jyear{2021}%

\newtheorem{theorem}{Theorem}
\newtheorem{lemma}[theorem]{Lemma}
\newtheorem{conjecture}[theorem]{Conjecture}
\newtheorem{claim}{Claim}

\newcommand{\D}{{\mathfrak{D}}}
\newcommand{\sB}{{\mathfrak{B}}}

\newcommand{\sP}{\mathcal{P}}

\newcommand{\PDD}{{principle of directional duality}}
\newcommand{\digon}{\leftrightarrow}
\newcommand{\lto}{\leftarrow}

\begin{document}

\title[$3$-anti-circulant digraphs are $\alpha$-diperfect and BE-diperfect]{$3$-anti-circulant digraphs are $\alpha$-diperfect and BE-diperfect}


\author*[1]{\fnm{Lucas Ismaily Bezerra} \sur{Freitas}}\email{ismailybf@ic.unicamp.br}

\author[1]{\fnm{Orlando} \sur{Lee}}\email{lee@ic.unicamp.br}
\equalcont{This author was supported by CNPq Proc. 303766/2018-2, CNPq Proc 425340/2016-3 and FAPESP Proc. 2015/11937-9. ORCID: 0000-0003-4462-3325.}

\affil[1]{\orgdiv{Institute of Computing}, \orgname{State University of Campinas}, \orgaddress{\street{Albert Einstein}, \city{Campinas}, \postcode{13083-852}, \state{São Paulo}, \country{Brazil}}}


\abstract{
Let $D$ be a digraph. A subset $S$ of $V(D)$ is a \emph{stable set} if every pair of vertices in $S$ is non-adjacent in $D$. A collection of disjoint paths $\sP$ is a \emph{path partition} of $V(D)$, if every vertex in $V(D)$ is in exactly one path of $\sP$. We say that a stable set $S$ and a path partition $\sP$ are \emph{orthogonal} if each path of $P$ contains exactly one vertex of $S$. A digraph $D$ satisfies the $\alpha$\emph{-property} if for every maximum stable set $S$ of $D$, there exists a path partition $\sP$ such that $S$ and $\sP$ are orthogonal. A digraph $D$ is $\alpha$\emph{-diperfect} if every induced subdigraph of $D$ satisfies the $\alpha$-property. In 1982, Claude Berge proposed a characterization for $\alpha$-diperfect digraphs in terms of forbidden \emph{anti-directed odd cycles}. In 2018, Sambinelli, Silva and Lee proposed a similar conjecture. A digraph $D$ satisfies the \emph{Begin-End-property} or \emph{BE-property} if for every maximum stable set $S$ of $D$, there exists a path partition $\sP$ such that (i)~$S$ and $\sP$ are orthogonal and (ii)~for each path $P\in\sP$, either the start or the end of $P$ belongs to $S$. A digraph $D$ is \emph{BE-diperfect} if every induced subdigraph of $D$ satisfies the BE-property. Sambinelli, Silva and Lee proposed a characterization for BE-diperfect digraphs in terms of forbidden \emph{blocking odd cycles}. In this paper, we verified both conjectures for $3$-anti-circulant digraphs. We also present some structural results for $\alpha$-diperfect and BE-diperfect digraphs.
}

\keywords{$3$-anti-circulant digraph, Diperfect digraph, Berge's conjecture, Begin-End conjecture}



\maketitle

\section{Notation}
\label{nota}

We assume that the reader is familiar with basic concepts of graph theory. Thus this section is mainly concerned with establishing the notation used. For definitions that are not present in this paper, we refer the reader to Bang-Jensen and Gutin's book~\cite{bang2008digraphs} or Bondy and Murty's book~\cite{Bondy08}. 

Let $D$ be a digraph with vertex set $V(D)$ and arc set $A(D)$. We only consider finite digraphs without loops and multiple arcs. Given two vertices $u$ and $v$ of $V(D)$, we denote an arc from $u$ to $v$ by $uv$. In this case, we say that $u$ \emph{dominates} $v$, and we denote this by $u \to v$. We say that $u$ and $v$ are \emph{adjacent} if $u \to v$ or $v \to u$; otherwise, we say that $u$ and $v$ are \emph{non-adjacent}. If $u \to v$ and $v \to u$, then we denote this by $u \digon v$; we also say that $\{u,v\}$ is a \emph{digon}. If every pair of distinct vertices of $D$ are adjacent, then we say that $D$ is a \emph{semicomplete digraph}. A digraph $H$ is a \emph{subdigraph} of $D$ if $V(H)\subseteq V(D)$ and $A(H) \subseteq A(D)$; moreover, if every arc of $A(D)$ with both vertices in $V(H)$ is in $A(H)$, then we say that $H$ is \emph{induced} by $X = V(H)$, and we write $H = D[X]$. If $uv$ is an arc of $D$, then we say that $u$ and $v$ are \emph{incident} in $uv$. We say that a digraph $H$ is \emph{inverse} of $D$ if $V(H) = V(D)$ and $A(H)= \{uv : vu \in A(D)\}$. The \emph{underlying graph} of $D$, denoted by $U(D)$, is the simple graph defined by $V(U(D))= V(D)$ and $E(U(D))= \{uv : u $ and $v$ are adjacent in $D\}$.

We say that a vertex $u$ is an \emph{in-neighbor} (resp., \emph{out-neighbor}) of a vertex $v$ if $u \to v$ (resp., $v \to u$). Let $X$ be a subset of $V(D)$. We denote by $N^-(X)$ (resp., $N^+(X)$) the set of vertices in $V(D)-X$ that are in-neighbors (resp., out-neighbors) of some vertex of $X$. We define the \emph{neighborhood} of $X$ as $N(X)=N^-(X) \cup N^+(X)$; when $X=\{v\}$, we write $N^-(v)$, $N^+(v)$ and $N(v)$. We say that $v$ is a \emph{source} if $N^-(v)=\emptyset$ and a \emph{sink} if $N^+(v)=\emptyset$.

For disjoint subsets $X$ and $Y$ of $V(D)$ (or subdigraphs of $D$), we say that $X$ and $Y$ are \emph{adjacent} if some vertex of $X$ and some vertex of $Y$ are adjacent. Moreover, $X \to Y$ means that every vertex of $X$ dominates every vertex of $Y$, $X \Rightarrow Y$ means that there exists no arc from $Y$ to $X$ and $X \mapsto Y$ means that both $X \to Y$ and $X \Rightarrow Y$ hold. When $X = \{x\}$ or $Y = \{y\}$, we write $x \mapsto Y$ and $X \mapsto  y$.

A \emph{path} $P$ in a digraph $D$ is a sequence of distinct vertices $P = v_1v_2 \dots v_k$ such that for all $v_i$ in $P$, $v_iv_{i+1} \in A(D)$ for $1 \leq i \leq k-1$. Whenever it is appropriate, we treat $P$ as being the subdigraph of $D$ with vertex set $V(P) = \{v_1 , v_2, \ldots , v_k\}$ and arc set $A(P)=\{v_iv_{i+1} : 1 \leq i \leq k-1\}$. We say that $P$ \emph{starts} at $v_1$ and \emph{ends} at $v_k$. We also say that $v_1,v_k$ are \emph{endvertices} of $P$ and $v_1$ is the \emph{initial} and $v_k$ is the \emph{final} of $P$; to emphasize this fact we may write $P$ as $v_1Pv_k$. Also, whenever it is convenient, we may omit the initial or the final in the notation as $v_1P$ or $Pv_k$. We denote by $v_iPv_j$ a \emph{subpath} of $P$ where $1 \leq i \leq j \leq k$. We define the \emph{length} of $P$ as $k-1$. We denote by $\overrightarrow{P_k}$ the class of isomorphism of a path of length $k-1$. If $V(P)=V(D)$, then we say that $P$ is a \emph{Hamilton path} of $D$, and in this case, we say that $D$ is \emph{traceable}. Let $P, Q$ be paths in $D$. If $P$ ends at some vertex $v$ and $Q$ starts at some vertex $u$ such that $v \to u$, then we denote by $PQ$ the \emph{concatenation} of $P$ and $Q$. We use this notation only if $PQ$ is a path.

A \emph{cycle} $C$ in $D$ is a sequence of vertices $C = v_1v_2 \dots v_kv_1$ such that $v_1v_2 \dots v_k$ is a path, $v_kv_1 \in A(D)$ and $k \geq 2$. Whenever it is convenient, we also treat $C$ as the subdigraph of $D$ with vertex set $V(C) = \{v_1 , v_2, \ldots , v_k\}$ and arc set $A(C)=\{v_iv_{i+1} : 1 \leq i \leq k\}$ where subscripts are taken modulo $k$. We define the \emph{length} of $C$ as $k$. If $k$ is odd, then we say that $C$ is an \emph{odd cycle}. We denote by $\overrightarrow{C_k}$ the class of isomorphism of a cycle of length $k$. If $V(C)=V(D)$, then we say that $C$ is a \emph{Hamilton cycle} of $D$, and we also say that $D$ is \emph{hamiltonian}. We say that $D$ is an \emph{acyclic digraph} if $D$ does not contain cycles. We also say that $C$ is a \emph{non-oriented cycle} if $C$ is not a cycle in $D$, but $U(C)$ is a cycle in $U(D)$. In particular, if a non-oriented cycle $C$ has length three, then we say that $C$ is a \emph{transitive triangle} in $D$.

Let $D$ be a digraph. A subset $S$ of $V(D)$ is a \emph{stable set} if every pair of vertices in $S$ is non-adjacent in $D$. The cardinality of a maximum stable set in $D$ is called the \emph{stability number} of $D$ and is denoted by $\alpha(D)$. A collection of disjoint paths $\sP$ of $D$ is a \emph{path partition} of $V(D)$, if every vertex in $V(D)$ belongs to exactly one path of $\sP$. Let $S$ be a stable set of $D$. We say that $S$ and $\sP$ are \emph{orthogonal} if $ \vert   V(P) \cap S \vert    = 1$ for every $P \in \sP$.

Let $G$ be a connected graph. A \emph{clique} is a set of pairwise adjacent vertices of $G$. The \emph{clique number} of $G$, denoted by $\omega(G)$, is the size of maximum clique of $G$. We say that a vertex set $B \subset V(G)$ is a \emph{vertex cut} if $G-B$ is a disconnected graph. If $G[B]$ is a complete graph, then we say that $B$ is a \emph{clique cut}. A \emph{(proper) coloring} of $G$ is a partition of $V(G)$ into stable sets $\{S_1,\ldots, S_k\}$. The \emph{chromatic number} of $G$, denoted by $\chi(G)$, is the cardinality of a minimum coloring of $G$. We say that $G$ is \emph{perfect} if for every induced subgraph $H$ of $G$, the equality $\omega(G)=\chi(H)$ holds. Moreover, we say that a digraph $D$ is \emph{diperfect} if $U(D)$ is perfect.

\section{Introduction}
\label{intro}

Some very important results in graph theory characterize a certain class of graphs (or digraphs) in terms of certain forbidden induced subgraphs (subdigraphs). The most famous one is probably Berge's Strong Perfect Graph Conjecture~\cite{berge1961}. Berge showed that neither an odd cycle of length at least five nor its complement is perfect. He conjectured that a graph $G$ is perfect if and only if it contains neither an odd cycle of length at least five nor its complement as an induced subdigraph. In 2006, Chudnovsky, Robertson, Seymour and Thomas~\cite{chudnovsky2006strong} proved Berge's conjecture, which became known as the Strong Perfect Graph Theorem.

\begin{theorem}[Chudnovsky, Robertson, Seymour and Thomas, 2006]
\label{perf}
A graph $G$ is perfect if and only if $G$ contains neither an odd cycle of length at least five nor its complement as an induced subgraph. 
\end{theorem}

In this paper we are concerned with two conjectures on digraphs which are somehow similar to Berge's conjecture. Those conjectures relate \emph{path partitions} and \emph{stable sets}. We need a few definitions in order to present both conjectures. 

Let $S$ be a stable set of a digraph $D$. An \emph{$S$-path partition} of $D$ is a path partition $\sP$ such that $S$ and $\sP$ are orthogonal. We say that $D$ satisfies the \emph{$\alpha$-property} if for every maximum stable set $S$ of $D$ there exists an $S$-path partition of $D$, and we say that $D$ is \emph{$\alpha$-diperfect} if every induced subdigraph of $D$ satisfies the $\alpha$-property. A digraph $C$ is an \emph{anti-directed odd cycle} if $({\rm i})$ $C = x_1x_2 \dots x_{2k+1}x_1$ is a non-oriented odd cycle, where $k \geq 2$ and $(\rm{ii})$ each of the vertices $x_1,x_2,x_3,x_4,x_6,x_8, \ldots, x_{2k}$ is either a source or a sink (see Figure~\ref{circ-berge}).

\begin{figure}[htbp]
\center
\subfloat[]{
    \tikzset{middlearrow/.style={
	decoration={markings,
		mark= at position 0.6 with {\arrow{#1}},
	},
	postaction={decorate}
}}

\tikzset{shortdigon/.style={
	decoration={markings,
		mark= at position 0.45 with {\arrow[inner sep=10pt]{<}},
		mark= at position 0.75 with {\arrow[inner sep=10pt]{>}},
	},
	postaction={decorate}
}}

\tikzset{digon/.style={
	decoration={markings,
		mark= at position 0.4 with {\arrow[inner sep=10pt]{<}},
		mark= at position 0.6 with {\arrow[inner sep=10pt]{>}},
	},
	postaction={decorate}
}}

\begin{tikzpicture}[scale = 0.5]		
	\node (n4) [black vertex] at (6.5,10) {};
	\node (n2) [black vertex] at (9,5)  {};
	\node (n3) [black vertex] at (9,8)  {};
	\node (n5) [black vertex] at (4,8)  {};
	\node (n1) [black vertex] at (4,5)  {};
	
	\node (label_n4)  at (6.5,10.5) {$v_4$};
	\node (label_n2)  at (9.5,4.5)  {$v_2$};
	\node (label_n3)  at (9.5,8.5)  {$v_3$};
	\node (label_n5)  at (3.5,8.5)  {$v_5$};
	\node (label_n1)  at (3.5,4.5)  {$v_1$};

  \foreach \from/\to in {n1/n2,n3/n2,n3/n4,n5/n4,n1/n5}
    \draw[edge,middlearrow={>}] (\from) -- (\to);    
\end{tikzpicture}
}
\quad
\subfloat[]{
    \tikzset{middlearrow/.style={
	decoration={markings,
		mark= at position 0.6 with {\arrow{#1}},
	},
	postaction={decorate}
}}

\tikzset{shortdigon/.style={
	decoration={markings,
		mark= at position 0.45 with {\arrow[inner sep=10pt]{<}},
		mark= at position 0.75 with {\arrow[inner sep=10pt]{>}},
	},
	postaction={decorate}
}}

\tikzset{digon/.style={
	decoration={markings,
		mark= at position 0.4 with {\arrow[inner sep=10pt]{<}},
		mark= at position 0.6 with {\arrow[inner sep=10pt]{>}},
	},
	postaction={decorate}
}}

\begin{tikzpicture}[scale = 0.5]		

	\node (n1) [black vertex] at (4,5)  {};
	\node (n2) [black vertex] at (9,5)  {};
	\node (n3) [black vertex] at (9,7)  {};
    \node (n4) [black vertex] at (9,9)  {};
	\node (n5) [black vertex] at (6.5,11) {};
	\node (n6) [black vertex] at (4,9)  {};
	\node (n7) [black vertex] at (4,7)  {};

	\node (label_n1)  at (3.5,4.5)  {$v_1$};
	\node (label_n2)  at (9.5,4.5)  {$v_2$};
	\node (label_n3)  at (9.5,7.5)  {$v_3$};
	\node (label_n4)  at (9.5,9.5)  {$v_4$};	
	\node (label_n5)  at (6.5,11.5) {$v_5$};
	\node (label_n6)  at (3.5,9.5)  {$v_6$};
	\node (label_n7)  at (3.5,7.5)  {$v_7$};

  \foreach \from/\to in {n1/n2,n3/n2,n1/n7,n6/n7,n6/n5,n5/n4,n3/n4}
    \draw[edge,middlearrow={>}] (\from) -- (\to);    
\end{tikzpicture}
}
\caption{\centering Examples of anti-directed odd cycles with length five and seven, respectively.}
\label{circ-berge}
\end{figure}
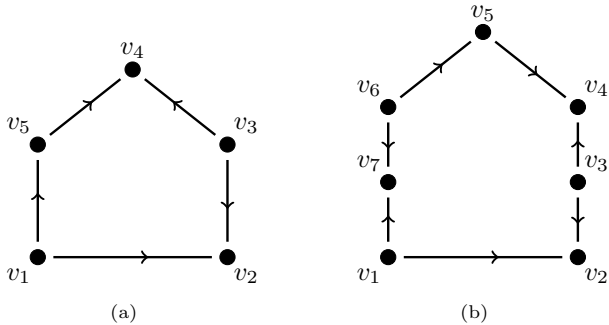

Berge~\cite{berge1981} showed that anti-directed odd cycles do not satisfy the $\alpha$-property, and hence, they are not $\alpha$-diperfect, which led him to conjecture the following characterization for $\alpha$-diperfect digraphs.

\begin{conjecture}[Berge, 1982]
\label{conj_berge}
A digraph $D$ is $\alpha$-diperfect if and only if $D$ does not contain an anti-directed odd cycle as an induced subdigraph. 
\end{conjecture}

Denote by $\sB$ the set of all digraphs which do not contain an induced anti-directed odd cycle. So Berge's conjecture can be stated as: $D$ is $\alpha$-diperfect if and only if $D$ belongs to $\sB$. In 1982, Berge~\cite{berge1981} verified Conjecture~\ref{conj_berge} for diperfect digraphs and for symmetric digraphs (digraphs such that if $uv \in A(D)$, then $vu \in A(D)$). In the next three decades, no results regarding this problem were published. In~\cite{tesemaycon2018,ssl}, Sambinelli, Silva and Lee verified Conjecture~\ref{conj_berge} for locally in-semicomplete digraphs and digraphs whose underlying graph is series-parallel. In~\cite{freitas2021BE}, Freitas and Lee verified Conjecture~\ref{conj_berge} for arc-locally (out) in-semicomplete digraphs. To the best of our knowledge, these papers are the only ones related to Conjecture~\ref{conj_berge} that were published recently.

In an attempt to understand the main difficulties in proving Conjecture~\ref{conj_berge}, Sambinelli, Silva and Lee~\cite{tesemaycon2018,ssl} introduced the class of Begin-End-diperfect digraphs, or simply BE-diperfect digraphs, which we define next.

Let $S$ be a stable set of a digraph $D$. A path partition $\sP$ is an \emph{$S_{BE}$-path partition} of $D$ if $(\rm{i})$ $\sP$ and $S$ are orthogonal and $(\rm{ii})$ every vertex of $S$ is the initial or the final of a path in $\sP$. We say that $D$ satisfies the \emph{BE-property} if for every maximum stable set of $D$ there exists an $S_{BE}$-path partition. We say that $D$ is \emph{BE-diperfect} if every induced subdigraph of $D$ satisfies the BE-property. Note that if $D$ is BE-diperfect, then it is also $\alpha$-diperfect, but the converse is not true (see the digraph in Figure~\ref{circ-bloqueante:b}). A digraph $C$ is a \emph{blocking odd cycle} if $(\rm{i})$ $C = x_1 x_2 \dots x_{2k+1}x_1$ is a non-oriented odd cycle, where $k \geq 1$ and $(\rm{ii})$ $x_1$ is a source and $x_2$ is a sink (see Figure~\ref{circ-bloqueante}). Note that every anti-directed odd cycle is also a blocking odd cycle.

\begin{figure}[htbp]
\center
\subfloat[]{
	    \tikzset{middlearrow/.style={
	decoration={markings,
		mark= at position 0.6 with {\arrow{#1}},
	},
	postaction={decorate}
}}

\tikzset{shortdigon/.style={
	decoration={markings,
		mark= at position 0.45 with {\arrow[inner sep=10pt]{<}},
		mark= at position 0.75 with {\arrow[inner sep=10pt]{>}},
	},
	postaction={decorate}
}}

\tikzset{digon/.style={
	decoration={markings,
		mark= at position 0.4 with {\arrow[inner sep=10pt]{<}},
		mark= at position 0.6 with {\arrow[inner sep=10pt]{>}},
	},
	postaction={decorate}
}}

\begin{tikzpicture}[scale = 0.5,auto=left]
	\node (n4) [black vertex] at (6.5,10) {};
	\node (n2) [black vertex] at (9,5)  {};
	\node (n3) [black vertex] at (9,8)  {};
	\node (n5) [black vertex] at (4,8)  {};
	\node (n1) [black vertex] at (4,5)  {};
	
	\node (label_n4)  at (6.5,10.5) {$v_4$};
	\node (label_n2)  at (9.5,4.5)  {$v_2$};
	\node (label_n3)  at (9.5,8.5)  {$v_3$};
	\node (label_n5)  at (3.5,8.5)  {$v_5$};
	\node (label_n1)  at (3.5,4.5)  {$v_1$};

  \foreach \from/\to in {n1/n2,n3/n2,n3/n4,n1/n5}
    \draw[edge,middlearrow={>}] (\from) -- (\to);
    
  \foreach \from/\to in {n5/n4}
    \draw[edge,digon] (\from) -- (\to);      
\end{tikzpicture}
	     \label{circ-bloqueante:a}
}
\quad
\subfloat[]{
    \input{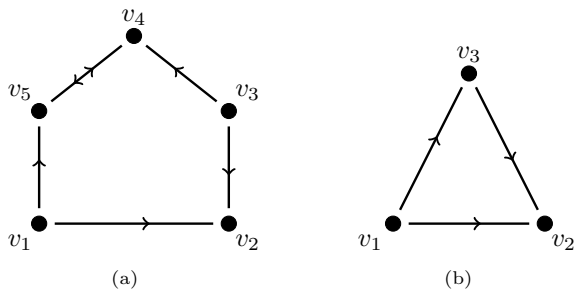}
    \label{circ-bloqueante:b}

}
\caption{\centering Examples of blocking odd cycles with length five and three, respectively. We also say that the digraph in (b) is a transitive triangle.}
\label{circ-bloqueante}
\end{figure}

Sambinelli, Silva and Lee~\cite{tesemaycon2018,ssl} showed that blocking odd cycles do not satisfy the BE-property, and hence, they are not BE-diperfect, which led them to conjecture the following characterization of BE-diperfect digraphs.

\begin{conjecture}[Sambinelli, Silva and Lee, 2018]
\label{conj_be}
A digraph $D$ is BE-diperfect if and only if $D$ does not contain a blocking odd cycle as an induced subdigraph.
\end{conjecture}

Denote by $\D$ the set of all digraphs which do not contain an induced blocking odd cycle. So Conjecture~\ref{conj_be} can be stated as: $D$ is BE-diperfect if and only if $D$ belongs to $\D$. Sambinelli, Silva and Lee~\cite{tesemaycon2018,ssl} verified Conjecture~\ref{conj_be} for locally in-semicomplete digraphs and digraphs whose underlying graph are series-parallel or perfect. In~\cite{freitas2021BE}, Freitas and Lee verified Conjecture~\ref{conj_be} for arc-locally (out) in-semicomplete digraphs. Note that a diperfect digraph belongs to $\D$ if and only if it contains no induced transitive triangle.

The rest of this paper is organized as follows. In Section~\ref{struc-results}, we present some structural results for $\alpha$-diperfect digraphs and BE-diperfect digraphs. In Section~\ref{3-anti-circ}, we present some structural results for $3$-anti-circulant digraphs and we verify both Conjecture~\ref{conj_berge} and Conjecture~\ref{conj_be} for these digraphs. In Section~\ref{conclu}, we present some final comments.

\section{Some structural results}
\label{struc-results}

In this section, we present some structural results for BE-diperfect digraphs and $\alpha$-diperfect digraphs. Let $D$ be a digraph and let $S$ be a maximum stable set of $D$. Since every $S_{BE}$-path partition of $D$ is also an $S$-path partition, it follows that if $D$ satisfies the BE-property, then $D$ also satisfies the $\alpha$-property. Moreover, the \emph{principle of directional duality} states that every structural result in a digraph has a companion structural result in its inverse digraph. Note that a digraph $D$ is BE-diperfect (resp., $\alpha$-diperfect) if and only if its inverse digraph is BE-diperfect (resp., $\alpha$-diperfect). 

Let us start with the following structural lemma.

\begin{lemma}
\label{arc-unique-N(u)-P-BE}
Let $D$ be a digraph such that every proper induced subdigraph of $D$ satisfies the BE-property (resp., $\alpha$-property). Let $S$ be a maximum stable set in $D$. Let $P=v_1v_2 \ldots v_k$ be a path of $D$ such that $V(P) \cap S = \emptyset$. If there exists a vertex $u$ in $D-V(P)$ such that $u \notin S$, $N^+(u) \neq \emptyset$ and $N^+(u) \subseteq V(P)$, then $D$ admits an $S_{BE}$-path partition (resp., $S$-path partition).
\end{lemma}

\begin{proof}
Let $i$ be the minimum in $\{1,2, \ldots,k \}$ such that $u \to v_i$. Let $P' = v_iP{v_k}$. Note that $N^+(u) \subseteq V(P')$. Let $D'=D-V(P')$. Note that $u$ is a sink in $D'$. Since $V(P') \cap S = \emptyset$, $S$ is a maximum stable set in $D'$. By hypothesis, $D'$ is BE-perfect. Let $\sP'$ be an $S_{BE}$-path partition of $D'$. Let $R$ be a path in $\sP'$ such that $u \in V(R)$. Since $u$ is a sink in $D'$, it follows that $R$ ends at $u$. Since $u \to v_i$, the collection $(\sP'-\{R\}) \cup \{RP'\}$ is an $S_{BE}$-path partition of $D$. 
\end{proof}

By the \PDD, we have the following result.

\begin{lemma}
\label{arc-unique-N(u)-P-BE-dual}
Let $D$ be a digraph such that every proper induced subdigraph of $D$ satisfies the BE-property (resp., $\alpha$-property). Let $S$ be a maximum stable set in $D$. Let $P=v_1v_2 \ldots v_k$ be a path of $D$ such that $V(P) \cap S = \emptyset$. If there exists a vertex $u$ in $D-V(P)$ such that $u \notin S$, $N^-(u) \neq \emptyset$ and $N^-(u) \subseteq V(P)$, then $D$ admits an $S_{BE}$-path partition (resp., $S$-path partition).
\qed
\end{lemma}

The next lemma is similar to Lemma~\ref{arc-unique-N(u)-P-BE}, but it provides a different technique.

\begin{lemma}
\label{arc-unique-u-disjoint-P-BE}
Let $D$ be a digraph such that every proper induced subdigraph of $D$ satisfies the BE-property (resp., $\alpha$-property). Let $S$ be a maximum stable set in $D$. Let $P=v_1v_2 \ldots v_k$ be a path of $D$ such that $V(P) \cap S = \emptyset$. If there exists an arc $u_1u_2$ in $A(D)$ such that $u_1 \notin S$, $\{u_1,u_2\} \cap V(P)= \emptyset$, $v_k \to u_2$ and $N^+(u_1) \subseteq V(P) \cup \{u_2\}$, then $D$ admits an $S_{BE}$-path partition (resp., $S$-path partition).
\end{lemma}

\begin{proof}
Let $i$ be the minimum in $\{1,2, \ldots,k \}$ such that $u_1 \to v_i$. Let $P' = v_iP{v_k}$. Note that $N^+(u_1) \subseteq V(P') \cup \{u_2\}$. Let $D'=D-V(P')$. Since $V(P') \cap S= \emptyset$, $S$ is a maximum stable set in $D'$. By hypothesis, $D'$ is BE-diperfect. Let $\sP'$ be an $S_{BE}$-path partition of $D'$. Let $R$ be a path in $\sP'$ such that $u_1 \in V(R)$. If $R$ ends at $u_1$, then since $u_1 \to v_i$, it follows that the collection $(\sP'-\{R\}) \cup \{RP'\}$ is an $S_{BE}$-path partition of $D$. So we may assume that $P$ does not end at $u_1$. Since $N^+(u_1) \subseteq V(P') \cup \{u_2\}$, it follows that $u_1u_2$ is an arc in $R$. Let $w_1$ and $w_p$ be the endvertices of $R$. Let $R_1 = w_1Ru_1$ and let $R_2 = u_2Rw_p$. Since $u_1 \to v_i$ and $v_k \to u_2$, the collection $(\sP'-\{R\}) \cup \{R_1P'R_2\}$ is an $S_{BE}$-path partition of $D$.
\end{proof}

By the \PDD, we have the following result.

\begin{lemma}
\label{arc-unique-u-disjoint-P-BE-dual}
Let $D$ be a digraph such that every proper induced subdigraph of $D$ satisfies the BE-property (resp., $\alpha$-property). Let $S$ be a maximum stable set in $D$. Let $P=v_1v_2 \ldots v_k$ be a path of $D$ such that $V(P) \cap S = \emptyset$. If there exists an arc $u_1u_2$ in $A(D)$ such that $u_2 \notin S$, $\{u_1,u_2\} \cap V(P)= \emptyset$, $u_1 \to v_1$ and $N^-(u_2) \subseteq V(P) \cup \{u_1\}$, then $D$ admits an $S_{BE}$-path partition (resp., $S$-path partition).
\qed
\end{lemma}

Next, we show that if a digraph $D$ contains a special partition of its vertices, then $D$ admits an $S_{BE}$-path partition (resp., $S$-path partition).

\begin{lemma}
\label{lem-part-V1-V2-V3}
Let $D$ be a digraph such that every proper induced subdigraph of $D$ satisfies the BE-property (resp.,$\alpha$-property). Let $S$ be a maximum stable set of $D$. If $V(D)$ admits a partition $(V_1,V_2,V_3)$ such that $V_1 \mapsto V_2 \mapsto V_3$, $D[V_2]$ is hamiltonian, $\vert V_2 \vert \ge 2$ and $ \vert V_2 \cap S \vert  \leq 1$, then $D$ admits an $S_{BE}$-path partition (resp., $S$-path partition).
\end{lemma}

\begin{proof}
Let $k =  \vert V_2 \vert $. Let $C=v_1v_2\ldots v_k$ be a Hamilton cycle in $D[V_2]$. Let $B$ be a subset of $V_2-S$ with cardinality $k-1$ (note that $B$ exists because $\vert V_2 \vert \geq 2$ and $ \vert V_2 \cap S \vert  \leq 1$). Without loss of generality, we may assume that $v_k$ is the vertex in $V_2-B$. Let $D' = D-B$. Since $B \cap S = \emptyset$, $S$ is maximum in $D'$. By hypothesis, $D'$ is BE-diperfect. Let $\sP'$ be an $S_{BE}$-path partition of $D'$. Let $P$ be a path in $\sP'$ such that $v_k \in V(P)$. First, suppose that $P$ does not start at $v_k$. Let $w$ be the vertex in $P$ that immediately precedes $v_k$. Let $P_1= Pw$ and let $P_2=v_kP$. Since $V_1 \mapsto V_2 \mapsto V_3$ and $V(D') \cap V_2=v_k$, it follows that $w$ in $V_1$. Let $R = v_1v_2\ldots v_{k-1}$. Since $V_1 \mapsto V_2$, it follows that $w \to v_1$. Since $v_{k-1} \to v_k$, we conclude that the collection $(\sP'-\{P\}) \cup \{P_1RP_2\}$ is an $S_{BE}$-path partition of $D$. So we may assume that $P$ starts at $v_k$. Let $w$ be the vertex in $P$ that immediately follows $v_k$. Let $P_1=v_k$ and let $P_2=wP$. Let $R = v_1v_2\ldots v_{k-1}$. Since $V_2 \mapsto V_3$, it follows that $v_{k-1} \to w$. Since $v_k \to v_1$, we conclude that the collection $(\sP'-\{P\}) \cup \{P_1RP_2\}$ is an $S_{BE}$-path partition of $D$.
\end{proof}

Next, we prove some lemmas to $\alpha$-diperfect digraphs.

\begin{lemma}
\label{3-anti-N(u)-S-alpha}
Let $D$ be a digraph such that every proper induced subdigraph of $D$ satisfies the $\alpha$-property. Let $S$ be a maximum stable set of $D$. Let $v_1v_2$ be an arc of $A(D)$. Then, 
\begin{enumerate}[label={(\roman*)},itemindent=1em]
	\item if $v_1 \notin S$ and $N^-(v_2)=\{v_1\}$, then $D$ admits an $S$-path partition,\label{3-anti-N(u)-S-alpha:i} 
	\item if $v_2 \notin S$ and $N^+(v_1)=\{v_2\}$, then $D$ admits an $S$-path partition.\label{3-anti-N(u)-S-alpha:ii} 
\end{enumerate}

\end{lemma}

\begin{proof}
By the \PDD, it suffices to prove \ref{3-anti-N(u)-S-alpha:i}. Let $D'=D-v_1$. Since $v_1 \notin S$, $S$ is a maximum stable set in $D'$. By hypothesis, $D'$ is $\alpha$-diperfect. Let $\sP'$ be an $S$-path partition of $D'$. Let $P$ be a path in $\sP'$ such that $v_2 \in V(P)$. Since $N^-(v_2)=\{v_1\}$, it follows that $P$ starts at $v_2$. Since $v_1 \to v_2$, the collection $(\sP'-\{P\}) \cup \{v_1P\}$ is an $S$-path partition of $D$.
\end{proof}

\begin{lemma}
\label{arc-unique-N(u)-P-alpha}
Let $D$ be a digraph such that every proper induced subdigraph of $D$ satisfies the $\alpha$-property. Let $S$ be a maximum stable set in $D$. Let $P=v_1v_2 \ldots v_k$, $k>1$, be a path of $D$ such that $(V(P)-\{v_1\}) \cap S =\emptyset$. If there exists a vertex $u$ in $D-V(P)$ such that $v_k \to u$ and $N^-(u) \subseteq V(P)$, then $D$ admits an $S$-path partition.
\end{lemma}

\begin{proof}
Let $P'= v_2Pv_k$. Let $D'=D-V(P')$. Since $V(P') \cap S = \emptyset$, $S$ is a maximum stable set in $D'$. By hypothesis, $D'$ is $\alpha$-diperfect. Let $\sP'$ be an $S$-path partition of $D'$. Let $R$ be a path in $\sP'$ such that $u \in V(R)$. Since $N^-(u) \subseteq V(P)$, it follows that $R$ starts at $u$ or $v_1u$ is an arc of $R$. If $P$ starts at $u$, then since $v_k \to u$, it follows that the collection $(\sP'-\{R\}) \cup \{P'R\}$ is an $S$-path partition of $D$. So suppose that $v_1u$ is an arc of $P$. Let $w_1$ and $w_p$ be the endvertices of $R$. Let $R_1 = w_1Rv_1$ and let $R_2=uRw_p$. Thus the collection $(\sP'-\{R\}) \cup \{R_1P'R_2\}$ is an $S$-path partition of $D$. 
\end{proof}

By the \PDD, we have the following result.

\begin{lemma}
\label{arc-unique-N(u)-P-alpha-dual}
Let $D$ be a digraph such that every proper induced subdigraph of $D$ satisfies the $\alpha$-property. Let $S$ be a maximum stable set in $D$. Let $P=v_1v_2 \ldots v_k$ be a path of $D$ such that $(V(P)-\{v_k\}) \cap S =\emptyset$. If there exists a vertex $u$ in $D-V(P)$ such that $u \to v_1$ and $N^+(u) \subseteq V(P)$, then $D$ admits an $S$-path partition.
\qed
\end{lemma}

\section{$3$-anti-circulant digraphs}
\label{3-anti-circ}

In this section, we verify both Conjecture~\ref{conj_berge} and Conjecture~\ref{conj_be} for $3$-anti-circulant digraphs which we define in this section.

Let $D$ be a digraph. We say that the set $\{v_1, v_2, v_3, v_4\} \subseteq V(D)$ is an \emph{anti-$P_4$} if $v_1 \to v_2$, $v_3 \to v_2$ and $v_3 \to v_4$. Whenever it is convenient, we may write an anti-$P_4$ as $v_1 \to v_2 \lto v_3 \to v_4$. Since every anti-directed odd cycle and every blocking odd cycle of length at least five contains an induced anti-$P_4$, it seems interesting to study digraphs that do not contain anti-$P_4 $ as an induced subdigraph. Motivated by this observation, we study the class of $3$-anti-circulant digraphs defined by Wang~\cite{wang2014} because they satisfy this property.

Let $D$ be a digraph. We say that $D$ is \emph{$3$-anti-circulant} if for every anti-$P_4$ $v_1 \to v_2 \lto v_3 \to v_4$ in $D$, it follows that $v_4 \to v_1$ (see Figure~\ref{fig-3-anti-circ:a}). Note that the inverse of $D$ is also a \emph{$3$-anti-circulant} digraph. So we can use the principle of directional duality whenever it is convenient. Moreover, note that every $3$-anti-circulant digraph belongs to $\sB$, and the only possible induced blocking odd cycle in a $3$-anti-circulant digraph is a transitive triangle (see Figure~\ref{fig-3-anti-circ:b}).
 
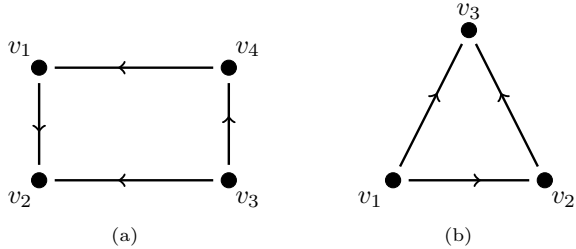
\begin{figure}[htbp]
\begin{center}
\subfloat[]{
       \tikzset{middlearrow/.style={
	decoration={markings,
		mark= at position 0.6 with {\arrow{#1}},
	},
	postaction={decorate}
}}

\tikzset{shortdigon/.style={
	decoration={markings,
		mark= at position 0.45 with {\arrow[inner sep=10pt]{<}},
		mark= at position 0.75 with {\arrow[inner sep=10pt]{>}},
	},
	postaction={decorate}
}}

\tikzset{digon/.style={
	decoration={markings,
		mark= at position 0.4 with {\arrow[inner sep=10pt]{<}},
		mark= at position 0.6 with {\arrow[inner sep=10pt]{>}},
	},
	postaction={decorate}
}}

\begin{tikzpicture}[scale = 0.5]		
	\node (n3) [black vertex] at (9,5)  {};
	\node (n4) [black vertex] at (9,8)  {};
	\node (n1) [black vertex] at (4,8)  {};
	\node (n2) [black vertex] at (4,5)  {};
	
	\node (label_n3)  at (9.5,4.5)  {$v_3$};
	\node (label_n4)  at (9.5,8.5)  {$v_4$};
	\node (label_n1)  at (3.5,8.5)  {$v_1$};
	\node (label_n2)  at (3.5,4.5)  {$v_2$};

  \foreach \from/\to in {n1/n2,n3/n2,n3/n4,n4/n1}
    \draw[edge,middlearrow={>}] (\from) -- (\to);    
\end{tikzpicture}
       \label{fig-3-anti-circ:a}

}
\quad
\subfloat[]{
       \tikzset{middlearrow/.style={
	decoration={markings,
		mark= at position 0.6 with {\arrow{#1}},
	},
	postaction={decorate}
}}

\tikzset{shortdigon/.style={
	decoration={markings,
		mark= at position 0.45 with {\arrow[inner sep=10pt]{<}},
		mark= at position 0.75 with {\arrow[inner sep=10pt]{>}},
	},
	postaction={decorate}
}}

\tikzset{digon/.style={
	decoration={markings,
		mark= at position 0.4 with {\arrow[inner sep=10pt]{<}},
		mark= at position 0.6 with {\arrow[inner sep=10pt]{>}},
	},
	postaction={decorate}
}}

\begin{tikzpicture}[scale = 0.5,auto=left]
  \node (n2) [black vertex] at (8,5)  {};
  \node (n3) [black vertex] at (6,9)  {};
  \node (n1) [black vertex] at (4,5)  {};
 
  \node (label_n2) at (8.5,4.5)  {$v_2$};
  \node (label_n3) at (6,9.5)  {$v_3$};
  \node (label_n1) at (3.4,4.5)  {$v_1$};

  \foreach \from/\to in {n1/n3,n2/n3,n1/n2}
    \draw[edge,middlearrow={>}] (\from) -- (\to);

\end{tikzpicture}
       \label{fig-3-anti-circ:b}

}

\caption{\centering Examples of $3$-anti-circulant digraphs.}
\label{fig-3-anti-circ}
\end{center}
\end{figure}

Moreover, Wang also characterized the structure of a strong $3$-anti-circulant digraph admitting a partition into vertex-disjoint cycles and showed that the structure is very close to semicomplete and semicomplete bipartite digraphs. This characterization does not seem to help in proving both conjectures for these digraphs. So we use a different approach. First, we need the following definitions.

Let $S$ be a maximum stable set of a digraph $D$. Denote by $B^+$ (resp., $B^-$) the subset of $V(D)-S$ such that $B \Rightarrow S$ (resp., $S \Rightarrow B$). Moreover, let $B^{\pm}=V(D)-(B^+ \cup B^- \cup S)$, that is, $B^{\pm}$ is a set of those vertices that both dominate and are dominated by some vertex in $S$ (see Figure~\ref{fig-3-anti-circ}). Note that $B^+$, $B^-$ and $B^{\pm}$ are pairwise disjoint and since $S$ is a maximum stable set in $D$, it follows that $V(D)= S \cup B^+ \cup B^- \cup B^{\pm}$.

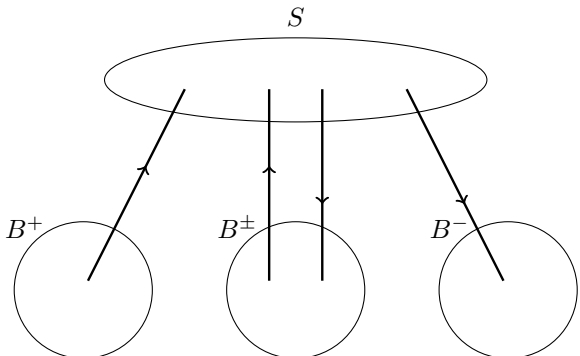
\begin{figure}[htbp]
\begin{center}
       \tikzset{middlearrow/.style={
	decoration={markings,
		mark= at position 0.6 with {\arrow{#1}},
	},
	postaction={decorate}
}}

\tikzset{shortdigon/.style={
	decoration={markings,
		mark= at position 0.45 with {\arrow[inner sep=10pt]{<}},
		mark= at position 0.75 with {\arrow[inner sep=10pt]{>}},
	},
	postaction={decorate}
}}

\tikzset{digon/.style={
	decoration={markings,
		mark= at position 0.4 with {\arrow[inner sep=10pt]{<}},
		mark= at position 0.6 with {\arrow[inner sep=10pt]{>}},
	},
	postaction={decorate}
}}

\begin{tikzpicture}[scale = 0.7,auto=left]
  
  \draw (4,5) circle (37pt); 
  \node (n1) at (4,5)  {};
    
  \draw (8,5) circle (37pt);  
  \node (n2) at (7.5,5)  {};
  \node (n3) at (8.5,5)  {};

  \draw (12,5) circle (37pt);  
  \node (n4)  at (12,5)  {};
  
  \draw (8,9) ellipse (3.6 and 0.8);
  \node (s1)  at (6,9)  {};
  \node (s2)  at (7.5,9)  {};
  \node (s3)  at (8.5,9)  {};
  \node (s4)  at (10,9)  {};

  \node (label_B+) at (2.9,6.2)  {$B^+$};
  \node (label_B+-) at (6.9,6.2)  {$B^{\pm}$};
  \node (label_B+) at (10.9,6.2)  {$B^-$};
  \node (label_S) at (8,10.2)  {$S$};

  \foreach \from/\to in {n2/s2,s3/n3}
    \draw[edge,middlearrow={>}] (\from) -- (\to);

  \foreach \from/\to in {n1/s1,s4/n4}
    \draw[edge,middlearrow={>}] (\from) -- (\to);

\end{tikzpicture}

\caption{\centering Illustration of $B^+$, $B^{\pm}$ and $B^-$.}
\label{fig-3-anti-circ}
\end{center}
\end{figure}

Let us start with a simple and useful structural lemma.

\begin{lemma}
\label{3-anti-circ-N(B+)<2}
Let $D$ be a $3$-anti-circulant digraph. Let $S$ be a maximum stable set in $D$. Then, for every $v$ in $B^+$ and for every $u$ in $B^-$, it follows that $\vert N^-(v) \cap B^+ \vert \leq 1$ and $ \vert N^+(u) \cap B^- \vert \leq 1$. 
\end{lemma}

\begin{proof}
Note that by the \PDD, it suffices to show that $\vert N^-(v) \cap B^+ \vert \leq 1$. Towards a contradiction, suppose that $ \vert N^-(v) \cap B^+ \vert > 1$. So let $v_1,v_2$ be vertices in $N^-(v) \cap B^+$. By definition of $B^+$, there exists a vertex $y$ in $S$ such that $v_1 \to y$. Since $v_2 \to v \lto v_1 \to y$ and $D$ is $3$-anti-circulant, it follows that $y \to v_2$, a contradiction because $v_2 \in B^+$. Thus $ \vert  N^-(v) \cap B^+ \vert   \leq 1$ and $ \vert  N^+(u) \cap B^- \vert   \leq 1$.
\end{proof}

\subsection{Begin-End conjecture}

In this subsection, we verify Conjecture~\ref{conj_be} for $3$-anti-circulant digraphs. In order to do this, we need the following auxiliary result by Freitas and Lee~\cite{freitas2021BE}.

\begin{lemma}[Freitas and Lee, 2021]
\label{stable_set_S_menor_vizinhanca}
Let $D$ be a digraph such that every proper induced subdigraph of $D$ satisfies the BE-property (resp., $\alpha$-property). If $D$ has a stable set $S$ such that $\vert N(S) \vert \leq \vert S \vert$, then $D$ satisfies the BE-property (resp., $\alpha$-property).
\end{lemma}

Initially, we present an outline of the main proof. Let $D$ be $3$-anti-circulant digraph and let $S$ be a maximum stable set in $D$. Note that every induced subdigraph of $D$ is also a $3$-anti-circulant digraph. Thus it suffices to show that $D$ satisfies the BE-property. First, we show that if $D \in \D$, then there exists no arc connecting vertices of distinct sets in $B^+$, $B^-$ and $B^{\pm}$. Next, we show that $B^+$, $B^-$ and $B^{\pm}$ are stable. This implies that $ \vert  S \vert   \geq  \vert  B^+ \cup B^- \cup B^{\pm} \vert  $, and hence, it follows by Lemma~\ref{stable_set_S_menor_vizinhanca} that $D$ satisfies the BE-property.

In the next three lemmas we show that if $U(D)$ contains a cycle $C$ of length three such that $C$ contains a digon and $V(C) \cap S \neq \emptyset$, then $D$ admits an $S_{BE}$-path partition.

\begin{lemma}
\label{3-anti-circ-digon-C3}
Let $D$ be a $3$-anti-circulant digraph such that every proper induced subdigraph of $D$ satisfies the BE-property. Let $S$ be a maximum stable set in $D$. Let $v_1 \digon v_2$ be a digon in $D-S$. If there exists a vertex $v_3$ in $V(D)-\{v_1,v_2\}$ such that $v_3 \in S$ and $D[\{v_1,v_2,v_3\}]$ contains a $\overrightarrow{C_3}$, then $D$ admits an $S_{BE}$-path partition.
\end{lemma}

\begin{proof}
With loss of generality, assume that $v_2 \to v_3$ and $v_3 \to v_1$. Let $D'=D-\{v_1,v_2\}$. Since $\{v_1,v_2\} \cap S = \emptyset$, $S$ is a maximum stable set in $D'$. By hypothesis, $D'$ is BE-diperfect. Let $\sP'$ be an $S_{BE}$-path partition of $D'$. Let $P$ be a path in $\sP'$ such that $v_3 \in V(P)$. If $V(P)=\{v_3\}$, then the collection $(\sP'-\{v_3\}) \cup \{v_1v_2v_3\}$ is an $S_{BE}$-path partition of $D$. So we may assume that $ \vert  V(P) \vert  >1$. By the principle of directional duality, we may assume that $P$ starts at $v_3$. Let $P = v_3w_1w_2 \ldots w_k$. Next, we show by induction on $k$ that $w_k \to v_1$ or $w_k \to v_2$ holds. First, suppose that $k=1$. Since $v_2 \to v_1 \lto v_3 \to w_1$ and $D$ is $3$-anti-circulant, it follows that $w_1 \to v_2$. Now, assume that $k>1$. By induction hypothesis, $w_{i-1} \to v_1$ or $w_{i-1} \to v_2$ for some $i \in \{2,\ldots, k\}$. Since $v_1 \digon v_2$ and $w_{i-1} \to w_i$, it follows that $w_i \to v_1$ or $w_i \to v_1$. Thus $w_k \to v_1$ or $w_k \to v_2$. Since $v_1 \digon v_2$, the collection $(\sP'-\{P\}) \cup \{Pv_1v_2\}$ or $(\sP'-\{P\}) \cup \{Pv_2v_1\}$ is an $S_{BE}$-path partition of $D$.       
\end{proof}

From now on, we prove some results for $3$-anti-circulant digraphs that belong to $\D$.

\setcounter{claim}{0}
\begin{lemma}
\label{3-anti-circ-digon-TT}
Let $D$ be a $3$-anti-circulant digraph such that every proper induced subdigraph of $D$ satisfies the BE-property. Let $S$ be a maximum stable set in $D$. Let $v_1 \digon v_2$ be a digon in $D$. If $D \in \D$ and there exists a vertex $v_3$ in $V(D)-\{v_1,v_2\}$ such that $\{v_1,v_2\} \to v_3$ and $\{v_1,v_2,v_3\} \cap S \neq \emptyset$, then $D$ admits an $S_{BE}$-path partition.
\end{lemma}

\begin{proof}
The proof is divided into two cases depending on whether $v_3 \in S$ or $v_3 \notin S$. First, we prove the following claim.

\begin{claim}\label{3anti:claim1}
If there exists a vertex $v_4 \in V(D)-\{v_1,v_2,v_3\}$ such $v_4 \to v_3$, then $D[\{v_1,v_2,v_3\}]$ is a complete digraph. 
\end{claim}

Since $\{v_1,v_2\} \to v_3$, $v_1 \digon v_2$ and $D$ is $3$-anti-circulant, it follows that $\{v_1,v_2\} \to v_4$. Since $v_2 \to v_4 \lto v_1 \to v_3$, we conclude that $v_3 \to v_2$, and hence, $v_2 \digon v_3$. Since $v_1 \to v_4 \lto v_2 \to v_3$, it follows that $v_3 \to v_1$, and hence, $v_1 \digon v_3$. Thus $D[\{v_1,v_2,v_3\}]$ is a complete digraph. This ends the proof of Claim~\ref{3anti:claim1}. \newline

\textbf{Case 1.} $v_3 \notin S$. If $N^-(v_3) \neq \{v_1,v_2\}$, then it follows by Claim~\ref{3anti:claim1} that $D[\{v_1,v_2,v_3\}]$ is complete, and hence, the result follows by Lemma~\ref{3-anti-circ-digon-C3}. So $N^-(v_3)= \{v_1,v_2\}$. If $v_2 \in S$ (resp., $v_1 \in S$), then since $N^-(v_3)= \{v_1,v_2\}$ and $v_3 \notin S$, the result follows by Lemma~\ref{arc-unique-u-disjoint-P-BE-dual} with $u_1=v_1$ (resp., $u_1 =v_2$), $u_2=v_3$ and $P=v_2$ (resp., $P=v_1$). \newline

\textbf{Case 2.} $v_3 \in S$. Since $\{v_1,v_2\} \to v_3$, $\{v_1,v_2\} \cap S = \emptyset$. We may assume by Lemma~\ref{3-anti-circ-digon-C3} that $v_1 \mapsto v_3$ and $v_2 \mapsto v_3$. Thus it follows by Claim~\ref{3anti:claim1} that $N^-(v_3)= \{v_1,v_2\}$. First, suppose that there exists a vertex $v_4$ in $N^+(v_2)-\{v_1,v_3\}$. Since $v_1 \to v_3 \lto v_2 \to v_4$, it follows that $v_4 \to v_1$. Since $v_4 \to v_1 \lto v_2 \to v_3$, we conclude that $v_3 \to v_4$. Since $D \in \D$, there exists at least one digon in $D[\{v_2,v_3,v_4\}]$; otherwise, $D[\{v_2,v_3,v_4\}]$ is an induced transitive triangle. Since $v_2 \mapsto v_3$ and $N^-(v_3)=\{v_1,v_2\}$, it follows that $v_2 \digon v_4$. Thus the result follows by Lemma~\ref{3-anti-circ-digon-C3} applied to $D[\{v_2,v_3,v_4\}]$. So we may assume that $N^+(v_2)=\{v_1,v_3\}$. Let $P = v_1$. Since $v_2 \notin S$, $\{v_2,v_3\} \cap V(P)= \emptyset$, $v_2 \to v_1$, $v_1 \to v_3$ and $N^+(v_2) \subseteq V(P) \cup \{v_3\}$, the result follows by Lemma~\ref{arc-unique-u-disjoint-P-BE} with $u_1=v_2$ and $u_2=v_3$. This finishes the proof.
\end{proof}

By the \PDD, we have the following result. 
 
\begin{lemma}
\label{3-anti-circ-digon-TT-dual}
Let $D$ be a $3$-anti-circulant digraph such that every proper induced subdigraph of $D$ satisfies the BE-property. Let $S$ be a maximum stable set in $D$. Let $v_1 \digon v_2$ be a digon in $D$. If $D \in \D$ and there exists a vertex $v_3$ in $V(D)-\{v_1,v_2\}$ such that $v_3 \to \{v_1,v_2\}$ and $\{v_1,v_2,v_3\} \cap S \neq \emptyset$, then $D$ admits an $S_{BE}$-path partition.
\qed
\end{lemma}

The following lemma states that we may assume that for every transitive triangle $T$ in $D \in \D$, $V(T) \cap S = \emptyset$.

\begin{lemma}
\label{3-anti-circ-has-TT-BE}
Let $D$ be a $3$-anti-circulant digraph such that every proper induced subdigraph of $D$ satisfies the BE-property. Let $S$ be a maximum stable set in $D$. If $D \in \D$ and $D$ contains a transitive triangle $T$ such that $V(T) \cap S \neq \emptyset$, then $D$ admits an $S_{BE}$-path partition.
\end{lemma}

\begin{proof}
Let $V(T)=\{v_1,v_2,v_3\}$. Without loss of generality, assume that $v_1 \to v_2$ and $\{v_1,v_2\} \to v_3$. Since $D \in \D$, there exists at least one digon in $T$; otherwise, $T$ is an induced transitive triangle. If $v_1 \digon v_2$ (resp., $v_2 \digon v_3$), then the result follows by Lemma~\ref{3-anti-circ-digon-TT} (resp., Lemma~\ref{3-anti-circ-digon-TT-dual}). Thus $v_1 \digon v_3$. If $v_2 \in S$, then the result follows by Lemma~\ref{3-anti-circ-digon-C3}. So $\{v_1,v_3\} \cap S \neq \emptyset$. Without loss of generality, assume that $v_3 \in S$. We show next that $N^+(v_1)=\{v_2,v_3\}$. Suppose that there exists a vertex $v_4$ in $N^+(v_1)-\{v_2,v_3\}$. Since $v_2 \to v_3 \lto v_1 \to v_4$ and $D$ is $3$-anti-circulant, we conclude that $v_4 \to v_2$. Also, since $v_4 \to v_2 \lto v_1 \to v_3$, it follows that $v_3 \to v_4$. Thus the result follows by Lemma~\ref{3-anti-circ-digon-TT} applied to $D[\{v_1,v_3,v_4\}]$. So we may assume that $N^+(v_1)=\{v_2,v_3\}$. Let $P=v_2$. Since $v_1 \notin S$, $\{v_1,v_3\} \cap V(P)= \emptyset$, $v_1 \to v_2$, $v_2 \to v_3$ and $N^+(v_1) \subseteq V(P) \cup \{v_3\}$, the result follows by Lemma~\ref{arc-unique-u-disjoint-P-BE} with $u_1=v_1$ and $u_2=v_3$. This finishes the proof.
\end{proof}

The next lemma states that we may assume that $B^- \cup B^{\pm} \Rightarrow B^+$.

\begin{lemma}
\label{3-anti-circ-B-+RB+}
Let $D$ be a $3$-anti-circulant digraph such that every proper induced subdigraph of $D$ satisfies the BE-property. Let $S$ be a maximum stable set in $D$. If $D \in \D$ and there are vertices $v_1 \in B^+$ and $v_2 \in B^- \cup B^{\pm}$ such that $v_1 \to v_2$, then $D$ admits an $S_{BE}$-path partition.
\end{lemma}

\begin{proof}
By definition of $B^+$, there exists a vertex $y_1$ in $S$ such that $v_1 \to y_1$. By definition of $B^{\pm}$ and $B^-$, there exists a vertex $y_2$ in $S$ such that $y_2 \to v_2$. Towards a contradiction, suppose that $y_1 \neq y_2$. Since $y_2 \to v_2 \lto v_1 \to y_1$ and $D$ is $3$-anti-circulant, it follows that $y_2 \to y_1$, a contradiction because $S$ is stable. So $y_1=y_2$, and hence, the result follows by Lemma~\ref{3-anti-circ-has-TT-BE} applied to $D[\{v_1,v_2,y_1\}]$. 
\end{proof}

By the \PDD, we have the following result.

\begin{lemma}
\label{3-anti-circ-B-+RB+-dual}
Let $D$ be a $3$-anti-circulant digraph such that every proper induced subdigraph of $D$ satisfies the BE-property. Let $S$ be a maximum stable set in $D$. If $D \in \D$ and there are $v_1 \in B^+ \cup B^{\pm}$ and $v_2 \in B^-$ such that $v_1 \to v_2$, then $D$ admits an $S_{BE}$-path partition.
\qed
\end{lemma}

We show next that if $D \in \D$, then we may assume that $B^{\pm}$ is a stable set. 

\begin{lemma}
\label{3-anti-circ-B+-stable}
Let $D$ be a $3$-anti-circulant digraph such that every proper induced subdigraph of $D$ satisfies the BE-property. Let $S$ be a maximum stable set in $D$. If $D \in \D$ and $B^{\pm}$ is not stable, then $D$ admits an $S_{BE}$-path partition.
\end{lemma}

\begin{proof}
Let $v_1,v_2$ be adjacent vertices in $B^{\pm}$. Without loss of generality, assume that $v_1 \to v_2$. By definition of $B^{\pm}$, there are vertices $y_1,y_2$ in $S$ such that $v_1 \to y_1$ and $y_2 \to v_2$. Since $S$ is stable and $D$ is $3$-anti-circulant, it follows that $y_1=y_2$, and hence, the result follows by Lemma~\ref{3-anti-circ-has-TT-BE} applied to $D[\{v_1,v_2,y_1\}]$.
\end{proof}

The next lemma states that if $D$ contains an anti-$P_4$ disjoint from $S$, then $D$ admits an $S_{BE}$-path partition.

\begin{lemma}
\label{3-anti-circ-antiP-noS}
Let $D$ be a $3$-anti-circulant digraph such that every proper induced subdigraph of $D$ satisfies the BE-property. Let $S$ be a maximum stable set in $D$. If $D \in \D$ and $D$ contains an anti-$P_4$ disjoint from $S$, then $D$ admits an $S_{BE}$-path partition.
\end{lemma}

\begin{proof}
Let $\{v_1, v_2, v_3, v_4\} \subseteq V(D)$ be an anti-$P_4$ in $D$ such that $v_1 \to v_2 \lto v_3 \to v_4$. Since $D$ is 3-anti-circulant, we conclude that $v_4 \to v_1$. We show next that $v_2 \in B^+$ and $v_3 \in B^-$. Note that by the \PDD, it suffices to show that $v_3 \in B^-$. Moreover, we may assume by Lemma~\ref{3-anti-circ-B-+RB+} that $B^- \cup B^{\pm} \Rightarrow B^+$. Towards a contradiction, suppose that $v_3 \notin B^-$. Since $v_3 \notin S$, it follows that $v_3 \in B^+ \cup B^{\pm}$. If $v_3 \in B^+$, then since $v_3 \to \{v_2,v_4\}$ and $B^- \cup B^{\pm} \Rightarrow B^+$, we conclude that $\{v_2,v_4\} \subset B^+$. Since $v_4 \in B^+$, it follows that $v_1 \in B^+$, and hence, $ \vert  N^-(v_2) \cap B^+ \vert  >1$, a contradiction by Lemma~\ref{3-anti-circ-N(B+)<2}. If $v_3 \in B^{\pm}$, then since $v_3 \to v_4$, it follows by Lemma~\ref{3-anti-circ-B+-stable} that $v_4 \notin B^{\pm}$. By Lemma~\ref{3-anti-circ-B-+RB+-dual}, $v_4 \notin B^-$. So $v_4 \in B^+$. Since $v_4 \in B^+$ and $B^- \cup B^{\pm} \Rightarrow B^+$, it follows that $v_1 \in B^+$. By definition of $B^{\pm}$, there exists a vertex $y$ in $S$ such that $v_3 \to y$. Since $v_1 \to v_2 \lto v_3 \to y$, we conclude that $y \to v_1$, a contradiction because $v_1 \in B^+$. Thus $v_3 \in B^-$ and $v_2 \in B^+$.

Now, let $P_1= v_4v_1v_2$ and let $P_2 = v_3v_4v_1$. Towards a contradiction, suppose that $N^+(v_3) \not\subseteq V(P_1)$ and $N^-(v_2) \not\subseteq V(P_2)$. So let $w_1, w_2$ vertices such that $w_1$ in $N^-(v_2) - V(P_2)$ and $w_2$ in $N^+(v_3) - V(P_1)$. First, suppose that $w_1=w_2$. Since $D \in \D$, there exists at least one digon in $D[\{v_2,v_3,w_1\}]$; otherwise, $D[\{v_2,v_3,w_1\}]$ is an induced transitive triangle. Since $v_2 \in B^+$ and $B^- \cup B^{\pm} \Rightarrow B^+$, we conclude that $w_1 \digon v_3$, and since $v_3 \in B^-$, the result follows by Lemma~\ref{3-anti-circ-B-+RB+-dual}. So we may assume that $w_1 \neq w_2$ (see Figure~\ref{fig-3-anti-prove-P4:a}). 

\begin{figure}[htbp]
\begin{center}
\subfloat[]{
    \tikzset{middlearrow/.style={
	decoration={markings,
		mark= at position 0.6 with {\arrow{#1}},
	},
	postaction={decorate}
}}

\tikzset{shortdigon/.style={
	decoration={markings,
		mark= at position 0.45 with {\arrow[inner sep=10pt]{<}},
		mark= at position 0.75 with {\arrow[inner sep=10pt]{>}},
	},
	postaction={decorate}
}}

\tikzset{digon/.style={
	decoration={markings,
		mark= at position 0.4 with {\arrow[inner sep=10pt]{<}},
		mark= at position 0.6 with {\arrow[inner sep=10pt]{>}},
	},
	postaction={decorate}
}}

\begin{tikzpicture}[scale = 0.7]		
	\node (n3) [black vertex] at (9,5)  {};
	\node (n4) [black vertex] at (9,8)  {};
	\node (n1) [black vertex] at (4,8)  {};
	\node (n2) [black vertex] at (4,5)  {};
	
	\node (n5) [black vertex] at (5,6.5)  {};
	\node (n6) [black vertex] at (8,6.5)  {};

	\node (label_n3)  at (9.5,4.5)  {$v_3$};
	\node (label_n4)  at (9.5,8.5)  {$v_4$};
	\node (label_n1)  at (3.5,8.5)  {$v_1$};
	\node (label_n2)  at (3.5,4.5)  {$v_2$};
	
	\node (label_n5)  at (4.5,6.8)  {$w_1$};
	\node (label_n6)  at (8.5,6.8)  {$w_2$};

  \foreach \from/\to in {n1/n2,n3/n2,n3/n4,n4/n1,n5/n2,n3/n6}
    \draw[edge,middlearrow={>}] (\from) -- (\to);    
\end{tikzpicture}
    \label{fig-3-anti-prove-P4:a}

}
\quad
\subfloat[]{
    \tikzset{middlearrow/.style={
	decoration={markings,
		mark= at position 0.6 with {\arrow{#1}},
	},
	postaction={decorate}
}}

\tikzset{shortdigon/.style={
	decoration={markings,
		mark= at position 0.45 with {\arrow[inner sep=10pt]{<}},
		mark= at position 0.75 with {\arrow[inner sep=10pt]{>}},
	},
	postaction={decorate}
}}

\tikzset{digon/.style={
	decoration={markings,
		mark= at position 0.4 with {\arrow[inner sep=10pt]{<}},
		mark= at position 0.6 with {\arrow[inner sep=10pt]{>}},
	},
	postaction={decorate}
}}

\begin{tikzpicture}[scale = 0.7]		
	\node (n3) [black vertex] at (9,5)  {};
	\node (n4) [black vertex] at (9,8)  {};
	\node (n1) [black vertex] at (4,8)  {};
	\node (n2) [black vertex] at (4,5)  {};
	
	\node (n5) [black vertex] at (5,6.5)  {};
	\node (n6) [black vertex] at (8,6.5)  {};

	\node (label_n3)  at (9.5,4.5)  {$v_3$};
	\node (label_n4)  at (9.5,8.5)  {$v_4$};
	\node (label_n1)  at (3.5,8.5)  {$v_1$};
	\node (label_n2)  at (3.5,4.5)  {$v_2$};
	
	\node (label_n5)  at (4.5,6.8)  {$w_1$};
	\node (label_n6)  at (8.5,6.8)  {$w_2$};

    \draw[edge,digon] (n1) -- (n4);    
	  
  \foreach \from/\to in {n1/n2,n3/n2,n3/n4,n5/n2,n3/n6,n6/n5,n4/n5,n6/n1}
    \draw[edge,middlearrow={>}] (\from) -- (\to);    
\end{tikzpicture}
    \label{fig-3-anti-prove-P4:b}

}

\caption{\centering Illustration for the proof of Lemma~\ref{3-anti-circ-antiP-noS}.}
\label{fig-3-anti-prove-P4}
\end{center}
\end{figure}
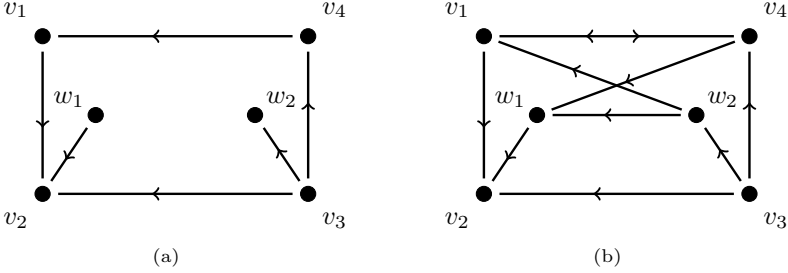

Since $w_1 \to v_2 \lto v_3 \to \{w_2,v_4\}$, we conclude that $\{w_2,v_4\} \to w_1$. Since $v_1 \to v_2 \lto v_3 \to w_2$, $w_2 \to v_1$. Also, since $v_4 \to w_1 \lto w_2 \to v_1$, we conclude that $v_1 \to v_4$, and hence, $v_1 \digon v_4$ (see Figure~\ref{fig-3-anti-prove-P4:b}). Since $v_3 \to v_4 \lto v_1 \to v_2$, it follows that $v_2 \to v_3$, a contradiction because $v_2 \in B^+$, $v_3 \in B^-$ and $B^- \cup B^{\pm} \Rightarrow B^+$. Thus $N^+(v_3) \subseteq V(P_1)$ or $N^-(v_2) \subseteq V(P_2)$. Since $\{v_1,v_2,v_3,v_4\} \cap S = \emptyset$, the result follows by Lemma~\ref{arc-unique-N(u)-P-BE} with $u=v_3$ or by Lemma~\ref{arc-unique-N(u)-P-BE-dual} with $u=v_2$. 
\end{proof}

In the next lemmas, we show that if $D \in \D$, then $B^+$ and $B^-$ are stable. To do this, we show that there exists no arc $v_1v_2$ in $D$ such that $v_1 \in B^+ \cup B^-$ and $v_2 \in B^{\pm}$.

\setcounter{claim}{0}
\begin{lemma}
\label{3-anti-circ-N-arc-B+B+-}
Let $D$ be a $3$-anti-circulant digraph such that every proper induced subdigraph of $D$ satisfies the BE-property. Let $S$ be a maximum stable set in $D$. If $D \in \D$ and there are adjacent vertices $v_1,v_2$ in $V(D)$ such that $v_1 \in B^+ \cup B^-$ and $v_2 \in B^{\pm}$, then $D$ admits an $S_{BE}$-path partition.
\end{lemma}

\begin{proof}
By the principle of directional duality, we may assume that $v_1 \in B^+$. Also, we may assume by Lemma~\ref{3-anti-circ-B-+RB+} that $B^- \cup B^{\pm} \Rightarrow B^+$. So $v_2 \mapsto v_1$. By definition of $B^+$, there exists a vertex $y_1$ in $S$ such that $v_1 \mapsto y_1$. By definition of $B^{\pm}$, there exists a vertex $y_2$ in $S$ such that $v_2 \to y_2$.

\begin{claim} \label{proof:claim1}
$N^-(v_1) \cap B^+ = \emptyset$.
\end{claim}

Towards a contradiction, suppose that there exists $v_3 \in B^+$ such that $v_3 \to v_1$. Since $v_3 \to v_1 \lto v_2 \to y_2$ and $D$ is $3$-anti-circulant, it follows that $y_2 \to v_3$, a contradiction by definition of $B^+$. Thus $N^-(v_1) \cap B^+ = \emptyset$. This finishes the proof of Claim~\ref{proof:claim1}. \newline

If $N^-(v_1) = \{v_2\}$, then since $\{v_1,v_2\} \cap S = \emptyset$, the result follows by Lemma~\ref{arc-unique-N(u)-P-BE-dual} with $P=v_2$ and $u=v_1$. So there exists a vertex $v_3$ in $N^-(v_1) - v_2$. By definition of $B^+$, $v_3 \notin S$. By Claim~\ref{proof:claim1}, $v_3 \in B^{\pm} \cup B^-$. The rest of proof is divided into two cases depending on whether $v_3 \in B^{\pm}$ or $v_3 \in B^{-}$ .\newline

\textbf{Case 1.} $v_3 \in B^{\pm}$. Recall that $v_2 \to y_2$ with $y_2 \in S$. Since $v_2 \in B^{\pm}$, we may assume by Lemma~\ref{3-anti-circ-B+-stable} that $v_2$ and $v_3$ are non-adjacent. By definition of $B^{\pm}$, there exists a vertex $y_3$ in $S$ such that $v_3 \to y_3$. Towards a contradiction, suppose that $y_3 = y_2$. Since $v_3 \to y_2 \lto v_2 \to v_1$, we conclude that $v_1 \to v_3$, a contradiction because $B^- \cup B^{\pm} \Rightarrow B^+$. So $y_3 \neq y_2$. Since $v_3 \to v_1 \lto v_2 \to y_2$, $y_2 \to v_3$. Also, since $v_2 \to v_1 \lto v_3 \to y_3$, $y_3 \to v_2$ (see Figure~\ref{fig-3-anti-circ-N-arc-B+B+-}).

\begin{figure}[htbp]
\begin{center}
    \tikzset{middlearrow/.style={
	decoration={markings,
		mark= at position 0.6 with {\arrow{#1}},
	},
	postaction={decorate}
}}

\tikzset{shortdigon/.style={
	decoration={markings,
		mark= at position 0.45 with {\arrow[inner sep=10pt]{<}},
		mark= at position 0.75 with {\arrow[inner sep=10pt]{>}},
	},
	postaction={decorate}
}}

\tikzset{digon/.style={
	decoration={markings,
		mark= at position 0.4 with {\arrow[inner sep=10pt]{<}},
		mark= at position 0.6 with {\arrow[inner sep=10pt]{>}},
	},
	postaction={decorate}
}}

\begin{tikzpicture}[scale = 0.7,auto=left]
  \draw (8,5) circle (37pt);  
  \node (n2) [black vertex] at (8,5.5)  {};
  \node (n3) [black vertex] at (8,4.5)  {};
  \draw (4,5) circle (37pt); 
  \node (n1) [black vertex] at (4,5)  {};
  
  \draw (10,8) ellipse (3.6 and 0.8);
  \node (n4) [black vertex] at (9,8)  {};
  \node (n5) [black vertex] at (11,8)  {};

  \node (label_n1) at (3.6,4.8)  {$v_1$};
  \node (label_B+) at (3.7,6.6)  {$B^+$};
  \node (label_B+-) at (7.6,6.6)  {$B^{\pm}$};
  \node (label_S) at (7,8.8)  {$S$};
  \node (label_n3) at (8.5,4.5)  {$v_3$};  
  \node (label_n2) at (7.8,5.9)  {$v_2$};
 
  \node (label_n4) at (9,8.4)  {$y_2$};
  \node (label_n5) at (11,8.4)  {$y_3$};

  \foreach \from/\to in {n2/n1,n3/n1,n2/n4,n3/n5,n4/n3,n5/n2}
    \draw[edge,middlearrow={>}] (\from) -- (\to);

\end{tikzpicture}

\caption{\centering Illustration for the proof of Lemma~\ref{3-anti-circ-N-arc-B+B+-}.}
\label{fig-3-anti-circ-N-arc-B+B+-}
\end{center}
\end{figure}

\begin{claim} \label{proof:claim1.1}
$N^-(\{y_2,y_3\}) \cap (B^- \cup B^{\pm}) = \{v_2,v_3\}$.
\end{claim}

By definition of $B^-$, $N^-(\{y_2,y_3\}) \cap B^- = \emptyset$. Towards a contradiction, suppose that there exists a vertex $v_4 \in B^{\pm}-\{v_2,v_3\}$ such that $v_4 \to y_i$ for some $i \in \{2,3\}$. Since $v_4 \to y_i \lto v_i \to v_1$, we conclude that $v_1 \to v_4$, a contradiction because $B^- \cup B^{\pm} \Rightarrow B^+$. So $N^-(\{y_2,y_3\}) \cap (B^- \cup B^{\pm}) = \{v_2,v_3\}$. This ends the proof of Claim~\ref{proof:claim1.1}. \newline

\begin{claim} \label{proof:claim1.2}
$N^+(\{v_2,v_3\})-S = \{v_1\}$. 
\end{claim}

Towards a contradiction, suppose that there exists $v_4 \in V(D)-(S \cup \{v_1\})$ such that $v_i \to v_4$ for some $i \in \{2,3\}$. Then, $\{v_1,v_2,v_3,v_4\}$ is an anti-$P_4$ disjoint from $S$, and hence, the result follows by Lemma~\ref{3-anti-circ-antiP-noS}. So we may assume that $N^+(\{v_2,v_3\})-S = \{v_1\}$. This finishes the proof of Claim~\ref{proof:claim1.2}. \newline

\begin{claim} \label{proof:claim1.3}
If there exists a vertex $v_4$ in $V(D)-(S \cup \{v_1,v_2,v_3\})$ such that $v_4 \to v_i$ for some $i \in \{2,3\}$, then $v_4 \in B^-$ and $N^+(v_4) = \{v_2,v_3\}$. Moreover, $N^-(\{v_2,v_3\})-S=\{v_4\}$.
\end{claim}

Without loss of generality, assume that $v_4 \to v_3$. Since $B^- \cup B^{\pm} \Rightarrow B^+$, $v_4 \notin B^+$. Since $\{v_2,v_3\} \subseteq B^{\pm}$, it follows by Lemma~\ref{3-anti-circ-B+-stable} that $v_4 \in B^-$ (see Figure~\ref{fig-3-anti-circ-N-arc-B+B+-2}). 

\begin{figure}[htbp]
\begin{center}
    \tikzset{middlearrow/.style={
	decoration={markings,
		mark= at position 0.6 with {\arrow{#1}},
	},
	postaction={decorate}
}}

\tikzset{shortdigon/.style={
	decoration={markings,
		mark= at position 0.45 with {\arrow[inner sep=10pt]{<}},
		mark= at position 0.75 with {\arrow[inner sep=10pt]{>}},
	},
	postaction={decorate}
}}

\tikzset{digon/.style={
	decoration={markings,
		mark= at position 0.4 with {\arrow[inner sep=10pt]{<}},
		mark= at position 0.6 with {\arrow[inner sep=10pt]{>}},
	},
	postaction={decorate}
}}

\begin{tikzpicture}[scale = 0.7,auto=left]
  \draw (8,5) circle (37pt);  
  \node (n2) [black vertex] at (8,5.5)  {};
  \node (n3) [black vertex] at (8,4.5)  {};
  \draw (4,5) circle (37pt); 
  \node (n1) [black vertex] at (4,5)  {};
  
  \draw (10,8) ellipse (3.6 and 0.8);
  \node (n4) [black vertex] at (9,8)  {};
  \node (n5) [black vertex] at (11,8)  {};
  \draw (12,5) circle (37pt);  
  \node (n6) [black vertex] at (12,5)  {};

  \node (label_n1) at (3.6,4.8)  {$v_1$};
  \node (label_B+) at (3.7,6.6)  {$B^+$};
  \node (label_B+-) at (7.6,6.6)  {$B^{\pm}$};
  \node (label_B+) at (11.7,6.6)  {$B^-$};
    
  \node (label_S) at (7,8.8)  {$S$};
  \node (label_n3) at (8.4,4.2)  {$v_3$};  
  \node (label_n2) at (7.8,5.9)  {$v_2$};
 
  \node (label_n4) at (9,8.4)  {$y_2$};
  \node (label_n5) at (11,8.4)  {$y_3$};
  \node (label_n1) at (12.4,4.8)  {$v_4$};

  \foreach \from/\to in {n2/n1,n3/n1,n2/n4,n3/n5,n4/n3,n5/n2,n6/n3}
    \draw[edge,middlearrow={>}] (\from) -- (\to);

\end{tikzpicture}

\caption{\centering Illustration for the proof of Lemma~\ref{3-anti-circ-N-arc-B+B+-}.}
\label{fig-3-anti-circ-N-arc-B+B+-2}
\end{center}
\end{figure}

By definition of $B^-$, $N^+(v_4) \cap S = \emptyset$. Now, we show that $N^+(v_4) \subseteq \{v_2,v_3\}$. First, suppose that $v_4 \to v_1$. Since $D \in \D$, there exists at least one digon in $D[\{v_1,v_3,v_4\}]$; otherwise, $D[\{v_1,v_3,v_4\}]$ is an induced transitive triangle. Since $B^- \cup B^{\pm} \Rightarrow B^+$, $v_3 \digon v_4$ which contradicts Claim~\ref{proof:claim1.2}. So $v_1 \notin N^+(v_4)$. Now, let $v_5$ be a vertex in $N^+(v_4)-\{v_2,v_3\}$. By definition of $B^-$ and since $v_4 \in B^-$, it follows that $v_5 \notin S$. Since $y_2 \to v_3 \lto v_4 \to v_5$, we conclude that $v_5 \to y_2$. Since $v_5 \to y_2 \lto v_2 \to v_1$, we conclude that $v_1 \to v_5$. Thus since $\{v_1,v_3,v_4,v_5\} \cap S = \emptyset$ and $v_1 \to v_5 \lto v_4 \to v_3$, the result follows by Lemma~\ref{3-anti-circ-antiP-noS}. So $N^+(v_4) \subseteq \{v_2,v_3\}$. If $N^+(v_4)=\{v_i\}$ for some $i \in \{2,3\}$, then it follows by Lemma~\ref{arc-unique-N(u)-P-BE} with $P=v_i$ and $u=v_4$ that $D$ admits an $S_{BE}$-path partition. Thus $N^+(v_4) = \{v_2,v_3\}$. Moreover, if $N^-(\{v_2,v_3\})-S \supset \{v_4\}$, then $D$ contains an anti-$P_4$ disjoint from $S$, and hence, the result follows by Lemma~\ref{3-anti-circ-antiP-noS}. Thus $N^-(\{v_2,v_3\})-S = \{v_4\}$. This ends the proof of Claim~\ref{proof:claim1.3}. \newline

\begin{claim} \label{proof:claim1.4}
If $N^-(\{v_2,v_3\})-S \neq \emptyset$, then $N^-(v_1)=\{v_2,v_3\}$. 
\end{claim}

Let $v_4$ be a vertex in $N^-(\{v_2,v_3\})-S$. It follows by Claim~\ref{proof:claim1.3} that $N^+(v_4)=\{v_2,v_3\}$ and $N^-(\{v_2,v_3\})-S=\{v_4\}$. Suppose that there exists a vertex $v_5$ in $N^-(v_1)- \{v_2,v_3\}$. By definition of $B^+$, $v_5 \notin S$. Since $v_5 \to v_1 \lto v_2 \to y_2$, $y_2 \to v_5$. Also, since $v_4 \to v_3 \lto y_2 \to v_5$, $v_5 \to v_4$. Since $\{v_1,v_3,v_4,v_5\} \cap S = \emptyset$ and $v_3 \to v_1 \lto v_5 \to v_4$, it follows by Lemma~\ref{3-anti-circ-antiP-noS} that $D$ admits an $S_{BE}$-path partition. So we may assume that $N^-(v_1)=\{v_2,v_3\}$. This ends the proof of Claim~\ref{proof:claim1.4}. \newline

The rest of proof is divided into two subcases depending on whether $N^-(\{v_2,v_3\})-S \neq \emptyset$ or $N^-(\{v_2,v_3\})-S = \emptyset$. \newline

\textbf{Subcase 1.} $N^-(\{v_2,v_3\})-S \neq \emptyset$. Let $v_4$ be a vertex in $N^-(\{v_2,v_3\})-S$. It follows by Claim~\ref{proof:claim1.3} that $N^+(v_4)=\{v_2,v_3\} $ and $N^-(\{v_2,v_3\})-S=\{v_4\}$. By Claim~\ref{proof:claim1.4}, $N^-(v_1)=\{v_2,v_3\}$. Let $D'=D-\{v_2,v_3\}$. Note that $v_1$ is a source and $v_4$ is a sink in $D'$. Since $\{v_2,v_3\} \cap S = \emptyset$, $S$ is a maximum stable set in $D'$. By hypothesis, $D'$ is BE-perfect. Let $\sP'$ be an $S_{BE}$-path partition of $D'$. Let $P_1,P_2$ be distinct paths in $\sP'$ such that $P_1$ starts at $v_1$ and $P_2$ ends at $v_4$. Thus the collection $(\sP'-\{P_1,P_2\}) \cup \{v_2P_1,P_2v_3\}$ is an $S_{BE}$-path partition of $D$. \newline

\textbf{Subcase 2.} $N^-(\{v_2,v_3\})-S= \emptyset$. By Claim~\ref{proof:claim1.2}, $N(\{v_2,v_3\})-S = \{v_1\}$. Let $D' = D-v_1$. Since $v_1 \notin S$, $S$ is a maximum stable set in $D'$. Let $\sP'$ be an $S_{BE}$-path partition of $D'$. Let $P_1$ be a path in $\sP'$ such that $v_2 \in V(P_1)$ and let $P_2$ be a path in $\sP'$ such that $v_3 \in V(P_2)$. In $D'$, $N(\{v_2,v_3\}) \subset S$. So it follows that both $P_1$ and $P_2$ have length one. If $P_1$ ends at $v_2$ or $P_2$ ends at $v_3$, then since $v_2 \to v_1$ and $v_3 \to v_1$, the collection $(\sP'-\{P_1\}) \cup \{P_1v_1\}$ or $(\sP'-\{P_2\}) \cup \{P_2v_1\}$ is an $S_{BE}$-path partition of $D$. Thus $P_1 = v_2w_1$ and $P_2 = v_3w_2$ with $w_1,w_2 \in S$. Since $\{v_2,v_3\} \to v_1$, $v_2 \to w_1$ and $v_3 \to w_2$, we conclude that $w_1 \to v_3$ and $w_2 \to v_2$. Thus the collection $(\sP'-\{P_1,P_2\}) \cup \{w_2v_2v_1,w_1v_2\}$ is an $S_{BE}$-path partition of $D$. \newline

\textbf{Case 2.} $v_3 \in B^-$. By definition of $B^-$, $N^+(v_3) \cap S = \emptyset$. If there exists a vertex $v_4$ in $N^+(v_3) - \{v_1,v_2\}$, then since $\{v_1,v_2,v_3,v_4\} \cap S = \emptyset$ and $v_2 \to v_1 \lto v_3 \to v_4$, the result follows by Lemma~\ref{3-anti-circ-antiP-noS}. Thus $N^+(v_3) \subseteq \{v_1,v_2\}$, and hence, since $\{v_1,v_2,v_3\} \cap S = \emptyset$, the result follows by Lemma~\ref{arc-unique-N(u)-P-BE} with $P=v_2v_1$ and $u=v_3$. This ends the proof. 

\end{proof}

Now, we show that if $D \in \D$, then we may assume that there exists no arc $v_1v_2$ in $D$ such that $v_1 \in B^+$ and $v_2 \in B^-$.

\begin{lemma}
\label{3-anti-circ-N-arc-B+B-}
Let $D$ be a $3$-anti-circulant digraph such that every proper induced subdigraph of $D$ satisfies the BE-property. Let $S$ be a maximum stable set of $D$. If $D \in \D$ and there are adjacent vertices $v_1,v_2$ in $V(D)$ such that $v_1 \in B^+$ and $v_2 \in B^-$, then $D$ admits an $S_{BE}$-path partition.
\end{lemma}

\begin{proof}
We may assume by Lemma~\ref{3-anti-circ-B-+RB+} that $B^- \cup B^{\pm} \Rightarrow B^+$. So $v_2 \mapsto v_1$. If $N^-(v_1) = \{v_2\}$, then since $\{v_1,v_2\} \cap S = \emptyset$, the result follows by Lemma~\ref{arc-unique-N(u)-P-BE-dual} with $P=v_2$ and $u=v_1$. So there exists a vertex $v_3$ in $N^-(v_1) - v_2$. Since $v_1 \in B^+$, $v_3 \notin S$. Since $v_2 \in B^-$, $N^+(v_2) \cap S =  \emptyset$. If there exists a vertex $v_4$ in $N^+(v_2)-\{v_1,v_3\}$, then since $\{v_1,v_2,v_3,v_4\} \cap S = \emptyset$ and $v_3 \to v_1 \lto v_2 \to v_4$, the result follows by Lemma~\ref{3-anti-circ-antiP-noS}. So we may assume that $N^+(v_2) \subseteq \{v_1,v_3\}$. Since $\{v_1,v_2,v_3\} \cap S = \emptyset$, it follows by Lemma~\ref{arc-unique-N(u)-P-BE} with $P=v_3v_1$ and $u=v_2$ that $D$ admits an $S_{BE}$-path partition. This finishes the proof.
\end{proof}

We show next that we may assume that $B^+ \cup B^-$ is a stable set.

\begin{lemma}
\label{3-anti-circ-N-arc-B+B-stable}
Let $D$ be a $3$-anti-circulant digraph such that every proper induced subdigraph of $D$ satisfies the BE-property. Let $S$ be a maximum stable set of $D$. If $D \in \D$ and $B^+ \cup B^-$ is not a stable set, then $D$ admits an $S_{BE}$-path partition.
\end{lemma}

\begin{proof}
If there are adjacent vertices $v_1,v_2$ in $V(D)$ such that $v_1 \in B^+$ and $v_2 \in B^-$, then the result follows by Lemma~\ref{3-anti-circ-N-arc-B+B-}. Let $v_1v_2$ be an arc in $D[B^+ \cup B^-]$. By the principle of directional duality, we may assume that $\{v_1,v_2\} \subseteq B^+$. Towards a contradiction, suppose that $N^-(v_2) \supset \{v_1\}$. Let $v_3$ be a vertex in $N^-(v_2)-v_1$. By definition of $B^+$, $v_3 \notin S$. Moreover, we may assume by Lemmas~\ref{3-anti-circ-N-arc-B+B-} and \ref{3-anti-circ-N-arc-B+B+-} that $v_3 \in B^+$. By definition of $B^+$, let $y$ be a vertex in $S$ such that $v_1 \to y$. Since $v_3 \to v_2 \lto v_1 \to y$, we conclude that $y \to v_3$, a contradiction by definition of $B^+$. Thus $N^-(v_2) = \{v_1\}$. Since $\{v_1,v_2\} \cap S = \emptyset$, it follows by Lemma~\ref{arc-unique-N(u)-P-BE-dual} with $P=v_1$ and $u=v_2$ that $D$ admits an $S_{BE}$-path partition. 
\end{proof}

Finally, we are ready for the main result of this subsection.

\begin{theorem}
\label{3-anti-circ-be}
Let $D$ be a $3$-anti-circulant digraph. If $D \in \D$, then $D$ is BE-diperfect.
\end{theorem}

\begin{proof}
Let $S$ be a maximum stable set of $D$. Since every induced subdigraph of $D$ is also a $3$-anti-circulant digraph, it suffices to show that $D$ satisfies the BE-property. Towards a contradiction, suppose the opposite and let $D$ be a counterexample with the smallest number of vertices. Note that if $D'$ is a proper induced subdigraph of $D$, then $D'$ is a $3$-anti-circulant digraph, and hence, by the minimality of $D$, it follows that $D'$ satisfies the BE-property. Thus $D$ does not satisfy the BE-property. It follows by Lemmas~\ref{3-anti-circ-B+-stable} and \ref{3-anti-circ-N-arc-B+B-stable} that both $B^{\pm}$ and $B^+ \cup B^-$ are stable. Thus it follows by Lemmas~\ref{3-anti-circ-N-arc-B+B+-} and \ref{3-anti-circ-N-arc-B+B-} that $B^+ \cup B^- \cup B^{\pm}$ is stable. Since $S$ is a maximum stable set of $D$, $ \vert  S \vert   \geq  \vert  B^+ \cup B^- \cup B^{\pm} \vert  $. Thus we conclude by Lemma~\ref{stable_set_S_menor_vizinhanca} that $D$ satisfies the BE-property, a contradiction. This ends the proof.
\end{proof}

\subsection{Berge's conjecture}

In this subsection, we verify Conjecture~\ref{conj_berge} for $3$-anti-circulant digraphs. Recall that every $3$-anti-circulant digraph belongs to $\sB$. The proof is divided into two cases depending on whether $D$ contains an induced transitive triangle or not.

\setcounter{claim}{0}
\begin{lemma}
\label{3-anti-circ-TT-berge}
Let $D$ be a $3$-anti-circulant digraph such that every proper induced subdigraph of $D$ satisfies the $\alpha$-property. If $D$ contains an induced transitive triangle $T$, then $D$ satisfies the $\alpha$-property. 
\end{lemma}

\begin{proof}
Let $S$ be a maximum stable set in $D$. Let $V(T)=\{v_1,v_2,v_3\}$. Without loss of generality, assume that $\{v_1,v_2\} \mapsto v_3$ and $v_1 \mapsto v_2$. First, we prove some claims. \newline

\begin{claim} \label{prof2:claim1}
$ \vert N^-(v_3) \vert  \leq 3$. Moreover, if there exists $v_4 \in N^-(v_3)-\{v_1,v_2\}$, then $v_4 \to v_1$ and $v_2 \to v_4$. 
\end{claim}

Towards a contradiction, suppose that there are distinct vertices $v_4,v_5$ in $N^-(v_3)-\{v_1,v_2\}$. Since $\{v_4,v_5\} \to v_3 \lto v_1 \to v_2$ and $D$ is $3$-anti-circulant, it follows that $v_2 \to \{v_4,v_5\}$. Since $v_1 \to v_3 \lto v_2 \to \{v_4,v_5\}$, we conclude that $\{v_4,v_5\} \to v_1$. Also, since $v_5 \to v_3 \lto v_2 \to v_4$, it follows that $v_4 \to v_5$. Now, since $v_2 \to v_5 \lto v_4 \to v_3$, it follows that $v_3 \to v_2$, and hence, $v_2 \digon v_3$, a contradiction because $v_2 \mapsto v_3$. Thus $ \vert N^-(v_3) \vert  \leq 3$. Moreover, note that if there exists $v_4 \in N^-(v_3)-\{v_1,v_2\}$, then $v_4 \to v_1$ and $v_2 \to v_4$. This ends the proof of Claim~\ref{prof2:claim1}.

\begin{claim} \label{prof2:claim2}
$\{v_1,v_2\} \cap S \neq \emptyset$.
\end{claim}

Suppose that $\{v_1,v_2\} \cap S = \emptyset$. First, suppose that there exists a vertex $v_4$ in $N^-(v_3) - \{v_1,v_2\}$. By Claim~\ref{prof2:claim1}, it follows that $N^-(v_3) = \{v_1,v_2,v_4\}$, $v_4 \to v_1$ and $v_2 \to v_4$. Let $D'=D-\{v_1,v_2\}$. Since $\{v_1,v_2\} \cap S = \emptyset$, $S$ is a maximum stable set in $D'$. By hypothesis, $D'$ is $\alpha$-diperfect. Let $\sP'$ be an $S$-path partition of $D'$. Let $P$ be a path in $\sP'$ such that $v_3 \in V(P)$. Since $N^-(v_3) = \{v_1,v_2,v_4\}$, it follows that $P$ starts at $v_3$ or $v_4v_3$ is an arc of $P$. If $P$ starts at $v_3$, then since $v_1 \to v_2$ and $v_2 \to v_3$, the collection $(\sP'-\{P\}) \cup \{v_1v_2P\}$ is an $S$-path partition of $D$ (note that if $N^-(v_3)=\{v_1,v_2\}$, then the result follows by previous argument). Thus $v_4v_3$ is an arc of $P$. Let $w_1$ and $w_p$ be the endvertices of $P$. Let $P_1=w_1Pv_4$ and $P_2=v_3Pw_p$ be the subpaths of $P$. Since $v_4 \to v_1$, $v_1 \to v_2$ and $v_2 \to v_3$, the collection $(\sP'-\{P\}) \cup \{P_1v_1v_2P_2\}$ is an $S$-path partition of $D$. So we may assume that $\{v_1,v_2\} \cap S \neq \emptyset$. This finishes the proof of Claim~\ref{prof2:claim2}. \newline

\begin{claim} \label{prof2:claim3}
$\{v_2,v_3\} \cap S \neq \emptyset$.
\end{claim}

By the \PDD, the result follows by Claim~\ref{prof2:claim2}. This ends the proof of Claim~\ref{prof2:claim3}. \newline

By Claims~\ref{prof2:claim2} and \ref{prof2:claim3}, it follows that $v_2 \in S$. First, suppose that there exists a vertex $v_4$ in $N^-(v_3)-\{v_1,v_2\}$. By Claim~\ref{prof2:claim1}, it follows that $N^-(v_3)=\{v_1,v_2,v_4\}$, $v_4 \to v_1$ and $v_2 \to v_4$. Let  $P=v_2v_4v_1$ and $u=v_3$. Since $(V(P)-v_2) \cap S = \emptyset$, $v_1 \to u$ and $N^-(u) \subseteq V(P)$, it follows by Lemma~\ref{arc-unique-N(u)-P-alpha} that $D$ admits an $S$-path partition. So we may assume that $N^-(v_3) = \{v_1,v_2\}$.

Now, suppose that $N^+(v_2)=\{v_3\}$. Since $v_3 \notin S$, the result follows by Lemma~\ref{3-anti-N(u)-S-alpha}\ref{3-anti-N(u)-S-alpha:ii}. So we may assume that there exists a vertex $w$ in $N^+(v_2)-\{v_1,v_3\}$. Since $v_1 \to v_3 \lto v_2 \to w$, we conclude that $w \to v_1$. Let $P=v_2wv_1$ and let $u=v_3$. Since $(V(P)-v_2) \cap S = \emptyset$, $v_1 \to u$ and $N^-(u) \subset V(P)$, the result follows by Lemma~\ref{arc-unique-N(u)-P-alpha}. This finishes the proof.
\end{proof}

We show next that if $D$ contains no induced transitive triangle, then $D$ satisfies the $\alpha$-property.

\begin{lemma}
\label{3-anti-circ-noTT-berge}
Let $D$ be a $3$-anti-circulant digraph such that every proper induced subdigraph of $D$ satisfies the $\alpha$-property. If $D$ contains no induced transitive triangle, then $D$ satisfies the $\alpha$-property.  
\end{lemma}

\begin{proof}
Since every blocking odd cycle of length at least five contains an induced anti-$P_4$ and $D$ is $3$-anti-circulant, it follows that $D$ contains no blocking odd cycle of length at least five. Moreover, $D$ contains no induced transitive triangle, and this implies that $D$ belongs to $\D$. So by Theorem~\ref{3-anti-circ-be} $D$ satisfies the BE-property, and hence, the $\alpha$-property.
\end{proof}

Now, we prove the main result of this subsection.

\begin{theorem}
\label{3-anti-circ-berge}
Let $D$ be a $3$-anti-circulant digraph. Then, $D$ is $\alpha$-diperfect.  
\end{theorem}

\begin{proof}
Since every induced subdigraph of $D$ is also a $3$-anti-circulant digraph, it suffices to show that $D$ satisfies the $\alpha$-property. If $D$ contains an induced transitive triangle, then the result follows by Lemma~\ref{3-anti-circ-TT-berge}. Thus $D$ contains no induced transitive triangle, and hence, the result follows by Lemma~\ref{3-anti-circ-noTT-berge}. This ends the proof.   
\end{proof}

\section{Concluding remarks}
\label{conclu}

In this paper, we presented two conjectures related to maximum stable set and path partition in digraphs. We verified both Conjectures~\ref{conj_berge} and \ref{conj_be} for $3$-anti-circulant digraphs. These digraphs do not contain anti-$P_4 $ as an induced subdigraph. We believe that study the structure these digraphs should help towards obtaining a proof of both conjectures in the general case.

Furthermore, an interesting and natural continuation in study of the structure of these digraphs is to analyze digraphs which for every anti-$P_4$ $v_1 \to v_2 \lto v_3 \to v_4$, it follows that $v_1$ and $v_4$ are adjacent. Here, we believe this could be a challenging problem.


\bibliography{ref}


\begin{thebibliography}{9}
\ifx \bisbn   \undefined \def \bisbn  #1{ISBN #1}\fi
\ifx \binits  \undefined \def \binits#1{#1}\fi
\ifx \bauthor  \undefined \def \bauthor#1{#1}\fi
\ifx \batitle  \undefined \def \batitle#1{#1}\fi
\ifx \bjtitle  \undefined \def \bjtitle#1{#1}\fi
\ifx \bvolume  \undefined \def \bvolume#1{\textbf{#1}}\fi
\ifx \byear  \undefined \def \byear#1{#1}\fi
\ifx \bissue  \undefined \def \bissue#1{#1}\fi
\ifx \bfpage  \undefined \def \bfpage#1{#1}\fi
\ifx \blpage  \undefined \def \blpage #1{#1}\fi
\ifx \burl  \undefined \def \burl#1{\textsf{#1}}\fi
\ifx \doiurl  \undefined \def \doiurl#1{\url{https://doi.org/#1}}\fi
\ifx \betal  \undefined \def \betal{\textit{et al.}}\fi
\ifx \binstitute  \undefined \def \binstitute#1{#1}\fi
\ifx \binstitutionaled  \undefined \def \binstitutionaled#1{#1}\fi
\ifx \bctitle  \undefined \def \bctitle#1{#1}\fi
\ifx \beditor  \undefined \def \beditor#1{#1}\fi
\ifx \bpublisher  \undefined \def \bpublisher#1{#1}\fi
\ifx \bbtitle  \undefined \def \bbtitle#1{#1}\fi
\ifx \bedition  \undefined \def \bedition#1{#1}\fi
\ifx \bseriesno  \undefined \def \bseriesno#1{#1}\fi
\ifx \blocation  \undefined \def \blocation#1{#1}\fi
\ifx \bsertitle  \undefined \def \bsertitle#1{#1}\fi
\ifx \bsnm \undefined \def \bsnm#1{#1}\fi
\ifx \bsuffix \undefined \def \bsuffix#1{#1}\fi
\ifx \bparticle \undefined \def \bparticle#1{#1}\fi
\ifx \barticle \undefined \def \barticle#1{#1}\fi
\bibcommenthead
\ifx \bconfdate \undefined \def \bconfdate #1{#1}\fi
\ifx \botherref \undefined \def \botherref #1{#1}\fi
\ifx \url \undefined \def \url#1{\textsf{#1}}\fi
\ifx \bchapter \undefined \def \bchapter#1{#1}\fi
\ifx \bbook \undefined \def \bbook#1{#1}\fi
\ifx \bcomment \undefined \def \bcomment#1{#1}\fi
\ifx \oauthor \undefined \def \oauthor#1{#1}\fi
\ifx \citeauthoryear \undefined \def \citeauthoryear#1{#1}\fi
\ifx \endbibitem  \undefined \def \endbibitem {}\fi
\ifx \bconflocation  \undefined \def \bconflocation#1{#1}\fi
\ifx \arxivurl  \undefined \def \arxivurl#1{\textsf{#1}}\fi
\csname PreBibitemsHook\endcsname

\bibitem{bang2008digraphs}
\begin{bbook}
\bauthor{\bsnm{{Bang-Jensen, J{\o}rgen}}},
\bauthor{\bsnm{{Gutin, Gregory Z.}}}:
\bbtitle{Digraphs: Theory, Algorithms and Applications}.
\bsertitle{Springer Monographs in Mathematics}.
\bpublisher{Springer},
\blocation{London}
(\byear{2008})
\end{bbook}
\endbibitem

\bibitem{Bondy08}
\begin{bbook}
\bauthor{\bsnm{{Bondy, J.A.}}},
\bauthor{\bsnm{{Murty, U.S.R.}}}:
\bbtitle{Graph Theory}.
\bsertitle{Graduate Texts in Mathematics},
vol. \bseriesno{244}.
\bpublisher{Springer},
\blocation{New York}
(\byear{2008})
\end{bbook}
\endbibitem

\bibitem{berge1961}
\begin{botherref}
\oauthor{\bsnm{Berge}, \binits{C.}}:
F{\"a}rbung von graphen, deren s{\"a}mtliche bzw. deren ungerade kreise starr
  sind.
Wissenschaftliche Zeitschrift
(1961)
\end{botherref}
\endbibitem

\bibitem{chudnovsky2006strong}
\begin{botherref}
\oauthor{\bsnm{{Chudnovsky, Maria}}},
\oauthor{\bsnm{{Robertson, Neil}}},
\oauthor{\bsnm{{Seymour, Paul}}},
\oauthor{\bsnm{{Thomas, Robin}}}:
The strong perfect graph theorem.
Annals of mathematics,
51--229
(2006)
\end{botherref}
\endbibitem

\bibitem{berge1981}
\begin{botherref}
\oauthor{\bsnm{Berge}, \binits{C.}}:
Diperfect graphs,
1--8
(1981)
\end{botherref}
\endbibitem

\bibitem{tesemaycon2018}
\begin{botherref}
\oauthor{\bsnm{{Sambinelli, Maycon}}},
\oauthor{\bsnm{{Lee, Orlando}}}:
Partition problems in graphs and digraphs.
PhD thesis,
University of Campinas - UNICAMP
(2018)
\end{botherref}
\endbibitem

\bibitem{ssl}
\begin{barticle}
\bauthor{\bsnm{{Sambinelli, Maycon}}},
\bauthor{\bsnm{{Silva, Cândida Nunes da}}},
\bauthor{\bsnm{{Lee, Orlando}}}:
\batitle{$\alpha$-diperfect digraphs}.
\bjtitle{Discrete Mathematics}
\bvolume{345}(\bissue{5}),
\bfpage{112759}
(\byear{2022})
\end{barticle}
\endbibitem

\bibitem{freitas2021BE}
\begin{botherref}
\oauthor{\bsnm{{Freitas, Lucas I. B.}}},
\oauthor{\bsnm{{Lee, Orlando}}}:
Some results on {B}erge's conjecture and {B}egin-{E}nd conjecture.
Submitted. arXiv: 2111.12168
(2021)
{\href{https://arxiv.org/abs/2111.12168}{{arXiv:2111.12168}}}
{[math.CO]}
\end{botherref}
\endbibitem

\bibitem{wang2014}
\begin{barticle}
\bauthor{\bsnm{Wang}, \binits{R.}}:
\batitle{Cycles in 3-anti-circulant digraphs}.
\bjtitle{Australasian Journal of Combinatorics}
\bvolume{60},
\bfpage{158}--\blpage{168}
(\byear{2014})
\end{barticle}
\endbibitem

\end{thebibliography}


\end{document}